\documentclass{article} 
\usepackage[T1]{fontenc}
\usepackage{indentfirst}
\usepackage{amsmath}
\usepackage{amssymb}
\usepackage{amsthm}
\usepackage{cancel}
\usepackage{dsfont}
\usepackage{stmaryrd}
\SetSymbolFont{stmry}{bold}{U}{stmry}{m}{n}
\usepackage{hyperref}
\usepackage{empheq}
\usepackage{cleveref}
\usepackage{bm}
\usepackage{enumitem}
\usepackage[pdftex]{graphicx}
\usepackage{fullpage}

\title{On the existence of a scalar pressure field in the Br\"odinger problem}
\author{Aymeric \textsc{Baradat}\footnote{CMLS, \'Ecole Polytechnique and \'Ecole Normale Sup\'erieure, France. \newline
E-mail: \emph{aymeric.baradat@polytechnique.edu}}}
\date{}

\newcommand{\R}{\mathbb{R}}
\newcommand{\Z}{\mathbb{Z}}
\newcommand{\T}{\mathbb{T}}
\newcommand{\E}{\mathbb{E}}
\newcommand{\N}{\mathbb{N}}

\newcommand{\A}{\mathcal{A}}
\renewcommand{\AA}{\boldsymbol{\mathcal{A}}}
\newcommand{\Fbf}{\boldsymbol{\mathcal{F}}}
\newcommand{\F}{\mathcal{F}}
\renewcommand{\H}{\mathcal{H}}
\newcommand{\Hbar}{\overline{\mathcal{H}}}
\newcommand{\HH}{\boldsymbol{\mathcal{H}}}
\renewcommand{\P}{\mathcal{P}}

\newcommand{\cg}{\langle}
\newcommand{\cd}{\rangle}

\newcommand{\pf}{_\#}
\newcommand{\eps}{\varepsilon}
\newcommand{\I}{\mathcal{I}}
\newcommand{\m}{\mathfrak{m}}
\newcommand{\rhorho}{\boldsymbol{\rho}}

\newcommand{\cc}{\boldsymbol{c}}

\newcommand{\Norm}{\mathsf{N}}

\newcommand{\Bro}{\mathsf{Br\ddot{o}}}

\newcommand{\MBro}{\mathsf{MBr\ddot{o}}}

\theoremstyle{plain}
\newtheorem{Thm}{Theorem}

\newtheorem{Lem}[Thm]{Lemma}

\theoremstyle{definition}
\newtheorem{Def}[Thm]{Definition}
\newtheorem{Rem}[Thm]{Remark}

\newtheorem{Pb}[Thm]{Problem}

\DeclareMathOperator{\Lip}{Lip}
\DeclareMathOperator{\Div}{div}
\DeclareMathOperator{\bDiv}{\bf{div}}

\DeclareMathOperator{\Leb}{Leb}

\DeclareMathOperator{\D}{d\!}
\DeclareMathOperator{\Id}{Id}
\DeclareMathOperator{\Diff}{D \!}

\begin{document} 
\maketitle
\begin{abstract}
In \cite{arn17}, Arnaudon, Cruzeiro, L\'eonard and Zambrini introduced an entropic regularization of the Brenier model for perfect incompressible fluids. We show that as in the original setting, there exists a scalar pressure field which is interpreted as the Lagrange multiplier associated to the incompressibility constraint. The proof goes through a reformulation of the problem in PDE terms.
\end{abstract}
\section*{Introduction}

\paragraph{Motivations.} In the seminal paper \cite{bre89}, Brenier introduces a variational problem - often called Brenier model for incompressible fluids or incompressible optimal transport - aiming at describing the evolution of an incompressible fluid inside a domain $D$ in a Lagrangian way, \emph{i.e.} by prescribing the state of the fluid at the initial and final times, and by minimizing an action functional in the set of dynamics that are incompressible and compatible with the initial and final prescription. This problem is a relaxed version of an older one studied by Arnol'd \cite{arn66} and Arnol'd-Khesin \cite{arn99} in order to derive the Euler equation for incompressible fluids from a least action principle: the problem of finding the geodesics of the formal infinite dimensional Lie group of all measure preserving diffeomorphisms of $D$ whose Lie algebra, which is identified as the set of divergence free vector fields on $D$, is endowed with the $L^2$ metric (see also \cite{ebi70}). 

To be more specific, in Brenier's relaxation, the admissible dynamics are chosen in the set $\P(C^0([0,1]; D))$\footnote{If $\mathcal{X}$ is a polish space, $\P(\mathcal{X})$ stands for the set of Borel probability measures on $\mathcal{X}$.} of \emph{generalized flows} on the physical domain $D$, and the goal is to minimize the functional that associates to a generalized flow $P$ the integral with respect to $P$ of the kinetic action of a single curve, under two constraints: one related to incompressibility, and one prescribing the law under $P$ of the endpoints of the curves. Not only this relaxation always admits solutions (see \cite{bre89}, in contrast with the smooth case \cite{shn87, shn94}), but it is in addition deeply linked to the hydrostatic approximation of the Euler equation \cite{brenier2008generalized}, and to a kinetic version of the Euler equation appearing as the \emph{quasineutral limit} of the Vlasov-Poisson equation \cite{Baradat2018wip}. We also refer non-exhaustively to \cite{bre93,bre99,amb09,ber09} for further studies of the Brenier model.

In this work, we will be interested in the \emph{entropic regularization} of this optimization problem, introduced by Arnaudon, Cruzeiro, L\'eonard and Zambrini \cite{arn17}, which consists in replacing the action functional of the Brenier model by the \emph{relative entropy} with respect to the law of the Brownian motion (all of this being properly defined later). The regularization is hence the same as the one transforming the optimal transport problem into its entropic regularized version, namely the Schr\"odinger problem in statistical mechanics, for which we refer for instance to \cite{sch31,sch32,zambrini1986variational,leonard2012schrodinger,leo13} and to the references therein. For this reason, as it is a mixture of the Brenier model and the Schr\"odinger problem, the authors of \cite{arn17} named this new model the \emph{Br\"odinger problem}. It is natural for at least two reasons. From a theoretical point of view, it links the incompressible optimal transport to the large deviation theory through Sanov's theorem \cite[Theorem~6.2.10]{dem09}. From a numerical point of view, Benamou, Carlier and Nenna \cite{benamou2017generalized} compute approximate solutions of incompressible optimal transport thanks to a time-discrete version of the Br\"odinger problem, using Sinkhorn algorithm \cite{sinkhorn1964relationship,sinkhorn1967diagonal}. Doing so, they extend to the incompressible case the techniques used to compute approximate solutions of the optimal transport problem \cite{cuturi2013sinkhorn,benamou2015iterative}. Let us mention that in \cite{baradat2018small}, written more or less at the same time as the present work, the author and Monsaingeon prove rigorously the convergence of the Br\"odinger problem towards the Brenier model, in the sense of $\Gamma$-convergence of the corresponding functionals. 

\paragraph{The pressure field.} One of the main results in the theory of incompressible optimal transport is the existence of a scalar pressure field only depending on the endpoints conditions (and not on the possibly non-unique solutions) acting as a Lagrange multiplier associated with the incompressibility constraint (see \cite{bre93} or \cite[Section 6]{amb09}), and a lot of attention have been dedicated to the study of its regularity \cite{ambfig08} (see also \cite{santambrogio2018regularity,lavenant2018new} for closely related results) and to its dependence with respect to the endpoints \cite{baradat2019continuous,Baradat2018wip}. 

This pressure field is proved to exist through the envelope theorem: we prove that if the incompressibility constraint is replaced by the prescription of a time dependent smooth density $\rho = \rho(t,x)$, then the optimal action varies at first order as:
\begin{equation*}
\cg p, \rho - 1 \cd,
\end{equation*}
for some scalar distribution $p \in \mathcal{D}'((0,1)\times D)$. In other words, the so-called pressure field $p$ is a distribution appearing as the G\^ateaux-differential of the optimal action when the density changes.

Besides, the pressure field is the relevant quantity when studying the dynamics of the paths charged by the solutions of incompressible optimal transport: it is shown in \cite[Theorem~6.8]{amb09} that if $P$ is a solution of incompressible optimal transport, then $P$ only charges trajectories $\omega$ that minimize for all $\eps \in (0,1/2)$ the classical Lagrangian:
\begin{equation*}
\int_\eps^{1-\eps} \left\{ \frac{|\dot \omega_t|^2}{2} - \overline p (t,\omega_t) \right\} \D t,
\end{equation*}
among the set of curves sharing their positions at times $\eps$ and $1-\eps$ with $\omega$, where $\overline p$ is one particular representative of $p$. All of this makes sense since $p$ is proved in \cite{ambfig08} to belong to some Lebesgue space of type $L^2_{\mbox{\footnotesize loc}}((0,1), L^q_{\mbox{\footnotesize loc}}(D))$ for some $q >1$. Otherwise stated, once the endpoints conditions are given, the incompressible optimal transport problem selects a pressure field, and all the particles follow the laws of classical mechanics corresponding to it.

For these reasons, it seems very natural to ask if it is possible to demonstrate in the Br\"odinger problem the existence of a similar scalar pressure field. The main result of this work is to bring a positive answer to this question. Before stating this result, let us give a precise formulation of the Br\"odinger problem.

\paragraph{Statement of the Br\"odinger problem.} We will work by convenience on the $d$-dimensional flat torus, \emph{i.e.} $D = \T^d := \R^d / \Z^d$. We will denote by $\Leb$ Lebesgue measure on $\T^d$, normalized so that $\Leb(\T^d)=1$, and by $R^\nu \in \P(C^0([0,1]; \T^d))$ the law of the reversible Brownian motion of diffusivity $\nu >0$ between the times $0$ and $1$, that is the Brownian motion whose marginal law at any time is $\Leb$.

Let $\gamma \in \P(\T^d \times \T^d)$. We will say that $\gamma$ is \emph{bistochastic} provided its first and second marginal is $\Leb$, which means that for all measurable $A \subset \T^d$:
\begin{equation*}
\gamma( A \times \T^d ) = \gamma ( \T^d \times A ) = \Leb(A).
\end{equation*}
These bistochastic measures will play the role of the prescribed joint laws at times $t=0$ and $1$ of the competitors in the Br\"odinger problem.

If $\mathcal{X}$ is a Polish space and if $p,r \in \P(\mathcal{X})$, the \emph{relative entropy} (or Kullback-Leibler divergence) of $p$ with respect to $r$ is defined by:\footnote{When it makes sense, we will use indifferently the notations $\int f \D p$ and $\E_p[f]$ for the integral of $f$ with respect to $p$.}
\begin{equation*}
H(p|r) := \left\{ 
\begin{aligned}
&\int \log \left( \frac{\D p}{\D r}\right) \D p = \E_p\left[ \log \left( \frac{\D p}{\D r}\right) \right] &&\mbox{if } p \ll r,\\
&+ \infty &&\mbox{else.}
\end{aligned}
\right.
\end{equation*}
It is well known that $r$ being fixed, $H(\bullet|r)$ is nonnegative, strictly convex, proper and lower semi-continuous with respect to the topology of narrow convergence \cite[Lemma~6.2.12]{dem09}. The Br\"odinger problem is defined as follows.
\begin{Pb}
	\label{pb.Bro}
	Given $\gamma$ bistochastic and $\nu >0$, find a generalized flow $P$ such that:
	\begin{enumerate}[label=(\alph*)]
		\item \label{item.compatible} The flow $P$ is compatible with $\gamma$, \emph{i.e.} for all measurable $A \subset \T^d \times \T^d$:
		\begin{equation}
		\label{eq.compatibility}
		P(\{\omega \mbox{ such that }(\omega_0,\omega_1) \in A \}) = \gamma(A).
		\end{equation}
		\item \label{item.incompressible} The flow $P$ is incompressible \emph{i.e.} for all measurable $A \subset \T^d$:
		\begin{equation}
		\label{eq.incompressibility}
		P(\{\omega \mbox{ such that }\omega_t \in A \}) = \Leb(A).
		\end{equation}
		\item The relative entropy:\footnote{The prefactor $\nu$ is the good scaling constant to get something of order one when $\nu \to 0$, see Theorem \ref{thm.girsanov} below.} 
		\begin{equation*}
		\Hbar_{\nu}(P) := \nu H(P|R^\nu)
		\end{equation*}
		is finite and minimal among the relative entropies of generalized flows satisfying \ref{item.compatible} and \ref{item.incompressible}.
	\end{enumerate}
	
	In the sequel, this problem will be referred to as $\Bro_\nu(\gamma)$. Note that calling $(X_t)_{t \in [0,1]}$ the canonical process on $C^0([0,1]; \T^d)$ and denoting by $\pf$ the \emph{push-forward} operator, condition \ref{item.compatible} can be written $(X_0,X_1)\pf P = \gamma$, and condition \ref{item.incompressible} can be written $\forall t \in [0,1], \, X_t {}\pf P = \Leb$.
\end{Pb}
Concerning the existence of solutions, it is proved in \cite[Corollary~5.2]{arn17} that $\Bro_\nu(\gamma)$ admits a solution if and only if:
\begin{equation}
\label{eq.condition_existence_process}
H(\gamma | \Leb \otimes \Leb) < +\infty.
\end{equation}
In that case, because of the strict convexity of the relative entropy, and in contrast with the incompressible optimal transport case, this solution is unique.

\paragraph{Statement of the main result.} Here is the main theorem we will prove in this work. 
\begin{Thm}
	\label{thm.existence_pressure_process}
	Take $\nu>0$ and $\gamma$ a bistochastic measure satisfying \eqref{eq.condition_existence_process}. Let $P$ be the solution of $\Bro_\nu(\gamma)$. There exists a unique scalar distribution $p \in \mathcal{D}'((0,1) \times \T^d)$ (up to adding a distribution only depending on time) such that for all $\varphi \in \mathcal{D}((0,1)\times \T^d)$ of zero spatial mean at all time, and all generalized flow $Q$ satisfying:
	\begin{itemize}
		\item $(X_0, X_1)\pf Q = \gamma$,
		\item for all $t \in [0,1]$, $X_t {}\pf Q = (1 + \varphi(t, \bullet)) \Leb$,
	\end{itemize}
	then:
	\begin{equation*}
	\Hbar_{\nu}(Q) \geq \Hbar_{\nu}(P) + \cg p, \varphi \cd.
	\end{equation*}
\end{Thm}
In fact, we will not prove this theorem directly, but we will rather state and prove a corresponding result in terms of PDEs, and then prove that these two results are indeed equivalent. The reason is that it is not so easy to modify the density of a generalized flow with finite entropy with respect to $R^\nu$ without losing this finite entropy condition (and hence to get estimates on the optimal entropy when the density changes). So let us now present the PDE framework.

\paragraph{The multiphase formulation of Br\"odinger.} As most of the models in the theory of optimal transport, the Br\"odinger problem comes with a Benamou-Brenier version (see \cite{ben00} for the classical case), \emph{i.e.} in terms of solutions to the continuity equation with prescribed initial and final densities. Here, as it was done by Brenier in the incompressible optimal transport case \cite{bre99} (see also \cite[Section~3 and 4]{amb09}), we will work with several phases of fluid described by their densities and macroscopic velocities denoted by $\rho^i = \rho^i(t,x)$ and $c^i(t,x)$ respectively, where $i$ belongs to a probability space of labels $(\I, \m)$. These phases will be coupled by an incompressibility constraint, meaning that for all $t \in [0,1]$:
\begin{equation*}
\int \rho^i(t, \bullet) \D \m(i) = \Leb.
\end{equation*}

As already observed by numerous papers \cite{chen2016relation,gentil2017analogy,gigli2018benamou}, the quantity to minimize in the entropic regularized framework is the kinetic action plus a penalizing term corresponding to the integral in time of the Fischer information. So we will work with solutions of the continuity equation with an additional regularity property, accordingly with the following definition.

\begin{Def}
	We say that the couple of density and velocity $(\rho, c)$ is a solution of the continuity equation provided:
\begin{itemize}
	\item  the density $\rho \in C^0([0,1]; \P(\T^d))$ and $c = c(t,x)$ is a measurable vector field with the following integrability:
		\begin{equation}
		\label{eq.def_action}
		\A(\rho, c) := \frac{1}{2}\int_0^1 \hspace{-5pt} \int |c(t,x)|^2 \rho(t, \D x) \D t < + \infty,
		\end{equation}
		\item the continuity equation:
		\begin{equation*}
		\partial_t \rho + \Div (\rho c) = 0
		\end{equation*}
		holds in a distributional sense.
\end{itemize}
We say in addition that this solution has finite Fischer information if for almost all time, $\rho$ has a density with respect to $\Leb$ that satisfies $\sqrt \rho \in L^2(0,1;H^1(\T^d))$. In that case, we write in the same way $\rho$ an its density, and we define $\nabla \log \rho$ through the identity:
\begin{equation*}
\frac{1}{2}\nabla \log \rho := \frac{\nabla \sqrt \rho}{\sqrt\rho},
\end{equation*}
which is well defined $\D t \otimes \rho$ almost everywhere. With this notation, $\sqrt \rho \in L^2(0,1;H^1(\T^d))$ translates into:
\begin{equation}
\label{eq.def_Fischer}
\F(\rho) := \frac{1}{2}\int_0^1 \hspace{-5pt} \int \left| \frac{1}{2} \nabla \log \rho (t,x)\right|^2 \rho(t, x) \D x \D t < + \infty.
\end{equation}
If $\sqrt \rho \notin L^2(0,1;H^1(\T^d))$, we just set $\F(\rho) := + \infty$.

Finally, we also define the following functional which will be seen as the Benamou-Brenier counterpart of the relative entropy with respect to the Brownian motion:
\begin{equation}
\label{eq.def_Hnu}
\H_\nu(\rho, c) := \frac{1}{2} \int_0^1 \hspace{-5pt} \int \left\{ |c(t,x)|^2 + \left| \frac{\nu}{2}\nabla \log \rho(t,x) \right|^2  \right\} \rho(t,x) \D x \D t = \A(\rho,c) + \nu^2 \F(\rho).
\end{equation}
\end{Def}

The multiphasic Br\"odinger problem is defined as follows.

\begin{Pb}
	\label{pb.Mbro}
	Given $\nu>0$, $(\I,\m)$ a probability space and $\rhorho_0 = (\rho_0^i)_{i \in \I}$,$\rhorho_1 = (\rho_1^i)_{i \in \I}$ two measurable families of probability measures on $\T^d$ satisfying:
	\begin{equation}
	\label{eq.incompressibility_t01}
		\int \rho^i_0 \D \m(i) =\int \rho^i_1 \D \m(i)= \Leb,
	\end{equation}
	find $(\rhorho, \cc)=(\rho^i, c^i)_{i\in \I}$ a measurable family of solutions of the continuity equation with finite Fischer information, well defined for $\m$-almost every $i$, such that:
		\begin{enumerate}[label=(\alph*)]
		\item \label{item.Mcompatibility} For $\m$-almost all $i \in \I$:
		\begin{equation*}
		\rho^i|_{t=0} = \rho^i_0 \qquad \mbox{and}\qquad \rho^i|_{t=1} = \rho^i_1.
		\end{equation*}
		\item \label{item.Mincompressibility} For all $t \in [0,1]$, we have:
		\begin{equation*}
		\int \rho^i(t,\bullet) \D \m(i) = \Leb.
		\end{equation*}
		\item The functional:
		\begin{equation}
		\label{eq.def_entropy_multiphase}
		\HH_\nu(\rhorho, \cc) :=  \int \H_\nu(\rho^i, c^i) \D \m(i) =: \AA(\rhorho, \cc) + \nu^2 \Fbf(\rhorho)
		\end{equation}
		(where $\AA(\rhorho)$ and $\Fbf(\rhorho)$ are the integrals with respect to $\m$ of $\A(\rho^i, c^i)$ and $\F(\rho^i)$ respectively) is finite and minimal among the measurable families of distributional solutions of the continuity equation satisfying \ref{item.Mcompatibility} and \ref{item.Mincompressibility}.
	\end{enumerate}
In the following, this problem will be called $\MBro_\nu(\rhorho_0, \rhorho_1)$.
\end{Pb}

As far as we know, this work and \cite{baradat2018small} are the first articles dealing with this optimization problem (even though in \cite{baradat2018small}, it is formulated using the notion of transport plans). A consequence of the analysis made in~\cite{baradat2018small} (see Theorem A.1) is that the existence of solution holds for this problem if and only if the initial and finite total entropies are finite, namely if and only if for $\m$-almost all $i$, $\rho^i_0$ and $\rho^i_1$ have densities with respect to $\Leb$, and:
\begin{equation}
\label{eq.CNS_existence_multiphase}
\int \hspace{-5pt}\int \rho^i_0(x) \log \rho^i_0(x) \D x \D \m(i) < + \infty \quad \mbox{and} \quad \int \hspace{-5pt}\int \rho^i_1(x) \log \rho^i_1(x) \D x \D \m(i) < +\infty.
\end{equation}
Once again, in that case, the solution is unique due to the strict convexity of $\H_\nu$ with respect to the variable $(\rho^i, \rho^ic^i)_{i \in \I}$.

As we already said, we will first prove a theorem corresponding to Theorem \ref{thm.existence_pressure_process} in the multiphasic setting, namely:
\begin{Thm}
	\label{thm.existence_pressure_MBro}
	Take $(\I, \m)$ a probability space, $\nu>0$, $\rhorho_0$ and $\rhorho_1$ satisfying \eqref{eq.incompressibility_t01}\eqref{eq.CNS_existence_multiphase}. Let $(\rhorho, \cc) = (\rho^i, c^i)_{i \in \I}$ be the solution of $\MBro_\nu(\rhorho_0, \rhorho_1)$.
	
There exists a unique scalar distribution $p \in \mathcal{D}'((0,1) \times \T^d)$ (up to adding a distribution only depending on time) such that for all $\varphi \in \mathcal{D}((0,1)\times \T^d)$ of zero spatial mean at all time, and for all measurable families $(\tilde{\rhorho}, \tilde{\cc}) = (\tilde{\rho}^i, \tilde{c}^i)_{i \in \I}$ of solutions of the continuity equation satisfying:
	\begin{enumerate}[label=(\roman*)]
		\item \label{item.t01} for $\m$-almost all $i$, $\tilde{\rho}^i|_{t = 0} = \rho^i_0$ et $\tilde{\rho}^i|_{t = 1} = \rho^i_1$,
		\item \label{item.density_varphi} for all $t \in [0,1]$, 
		\begin{equation*}
		\int \tilde{\rho}^i_t \D \m(i) = (1 + \varphi(t, \bullet)) \Leb,
		\end{equation*}
	\end{enumerate}
	Then:
	\begin{equation*}
	\HH_{\nu}(\tilde{\rhorho}, \tilde{\cc}) \geq \HH_{\nu}(\rhorho, \cc) + \cg p, \varphi \cd.
	\end{equation*}
	
	Moreover, the following formula holds in the sense of distribution:
	\begin{equation}
	\label{eq.formule_p_MBro}
	\partial_t \left( \int \rho^i c^i \D \m(i)\right) + \bDiv \left( \int \Big\{ c^i \otimes c^i -  w^i\otimes w^i \Big\}\rho^i \D \m(i) \right) = - \nabla p,
	\end{equation}
	where $w^i$, called the osmotic velocity of phase $i$, is defined by:
	\begin{equation*}
	w^i = \frac{\nu}{2} \nabla \log \rho^i.
	\end{equation*}
\end{Thm}
This time, the proof will more or less follow the one given in \cite[Theorem~6.2]{amb09}: We will prove that the optimal $\HH_\nu$ under constraints \ref{item.t01}\ref{item.density_varphi} is a convex function of $\varphi$ (Lemma \ref{lem.Hnustar_convex}), and that this function is bounded in some distributional neighborhood of $0$ (Lemma \ref{lem.Hnu_bounded}). We finally prove that its subdifferential at $0$ is a singleton, characterized by formula \eqref{eq.formule_p_MBro} (Lemma \ref{lem.charact_p}). (For the link between G\^ateaux-differentiability and subdifferential of a convex function, we refer to \cite{ekeland1999convex}.)

We will then transfer this result into Problem \ref{pb.Bro} by showing that every solution $P$ of the Br\"odinger problem gives rise to a solution to the multiphase version. We will be more precise in Section \ref{sec.lien_Bro_MBro} below, but what we will prove is that if $P$ is a solution of $\Bro_\nu(\gamma)$, then calling $P^{x,y} := P(\bullet | X_0 = x, X_1 = y)$, $\rho^{x,y}$ its density and $c^{x,y}$ its \emph{current velocity} (in the sense of \cite[Chapter 13]{nel67}), then up to localizing in time, $(\rho^{x,y}, c^{x,y})_{(x,y) \in \T^d \times \T^d}$ is a solution of the multiphase version of Br\"odinger, with $\I = \T^d \times \T^d$ and $\m = \gamma$.

\paragraph{Outline of the paper.} In Section \ref{sec.notations_chap5}, we introduce some notations, and give some preliminary results. In particular, we present in Subsection \ref{subsec.functional_spaces} the functional spaces we will work with in the rest of the paper. We also state a version of the Girsanov theorem which will be useful for showing the link between $\Bro$ and $\MBro$.

We prove Theorem \ref{thm.existence_pressure_MBro} in Section \ref{sec.existence_MBro}. It will be a consequence of Lemma \ref{lem.Hnustar_convex}, \ref{lem.Hnu_bounded} and \ref{lem.charact_p} that we already discussed.

Then, we develop the link between the problems $\Bro$ and $\MBro$, rigorously stated at Theorem \ref{thm.lien_Bro_MBro} of Section \ref{sec.lien_Bro_MBro}.

The proof of Theorem \ref{thm.existence_pressure_process}, given in Section \ref{sec.existence_Bro}, is a consequence of Theorem~\ref{thm.existence_pressure_MBro}, Theorem~\ref{thm.lien_Bro_MBro} and their respective proofs.

Finally, we give in Section \ref{section.formal_computations} a formal way to recover formula \eqref{eq.formule_p_MBro}, assuming by analogy with the non-viscous case that each phase is the solution of the Schr\"odinger problem corresponding to its endpoints, in the potential given by the pressure field. We do not prove that this condition is always verified, but it should be the "noisy" version of \cite[Theorem~6.8]{amb09} that we have already cited.

\section{Notations and preliminary results}
\label{sec.notations_chap5}

\subsection{Functional spaces of interest}
\label{subsec.functional_spaces}
In the whole chapter, if $\mathcal{B}$ is a Banach space, we denote by $\mathcal{B}'$ its topological dual. Two functional sets will be of particular interest.
\begin{enumerate}
	\item We will often consider the set $C^0([0,1], \P(\T^d))$ of curves on the set of probability measures on the torus. We endow it with the topology of uniform convergence corresponding to the topology of narrow convergence on $\P(\T^d)$. We write $M = (M_t)_{t \in [0,1]} \in C_0^0([0,1], \P(\T^d))$ whenever $M$ belongs to $C^0([0,1], \P(\T^d))$ and $M_0 = M_1 = \Leb$, the Lebesgue measure on $\T^d$.
	\item The set $\mathcal{E}$ will be the vector space of continuous scalar functions $f$ that satisfy 
	\begin{itemize}
		\item for all $t \in [0,1]$, $f(0,\bullet) = f(1,\bullet) = 0$, 
		\item for all $t \in [0,1]$, $f(t,\bullet) \in W^{2, \infty}(\T^d)$ and the Hessian $\mathrm{D}^2 f$ of $f$ satisfies:
		\[
		\sup_{t \in [0,1]} \| \mathrm{D}^2 f(t, \bullet) \|_{L_x^{\infty}} < + \infty,
		\]
		\item for all $x \in \T^d$, $f(\bullet , x) \in AC^2([0,1])$ and the temporal derivative $\partial_t f$, which is punctually defined for almost all $t \in [0,1]$ for almost all $x \in \T^d$, satisfies
		\[
		\int_0^1 \| \partial_t f(t) \|_{L^{\infty}(\T^d)}^2 \D t < + \infty.
		\]
	\end{itemize}
	On $\mathcal{E}$, we define the norm
	\begin{equation}
	\label{eq.def_N_Chap5}
	\forall f \in \mathcal{E}, \quad \Norm(f) := \sup_{t \in [0,1]} \| \mathrm{D}^2 f(t) \|_{L_x^{\infty}} + \left( \int_0^1 \| \partial_t f(t) \|_{L^{\infty}(\T^d)}^2 \D t \right)^{1/2}.
	\end{equation}
	
	If in addition, for all $t \in [0,1]$, $\int f(t, x) \D x = 0$, we write $f \in \mathcal{E}_0$. 
	
	In a slightly abusive way, we keep the same notations if $f$ has its values in $\R^d$.
	
	Remark that with these notations identifying a measure with its density with respect to $\Leb$, we have $C^0_0([0,1]; \P(\T^d)) \cap \mathcal{E} \subset 1 + \mathcal{E}_0$.
	
	In the proof of Theorem \ref{thm.existence_pressure_MBro}, the variations of densities ($\varphi$ in the statement of the theorem) will be studied in the topology of $\mathcal{E}_0$.
\end{enumerate}

\subsection{Preliminary results}
In order to get estimates on the optimal $\HH_\nu$ when the density varies, we will need to use maps that send $\Leb$ onto $(1+\varphi)\Leb$, for which we control enough derivatives (two with respect to space, one with respect to time). This will be possible thanks to the following theorem, which is a direct consequence of a famous result by Dacorogna and Moser \cite{dac90} in which the Monge-Amp\'ere equation is studied as a perturbation of the Poisson equation. We also refer to the Appendix of \cite{baradat2019continuous} for a short proof of this kind of result in the easy case when the domain is the torus.
\begin{Thm}
	\label{dacmos_chap5}
	Let $\varphi \in \mathcal{E}_0$ be such that $\Norm(\varphi) \leq 1/2$. There exists a dimensional constant $C>0$, $\xi = \xi(t,x) \in \R^d$ and $\zeta = \zeta(t,x) \in \R^d$ two vector fields of $\mathcal{E}_0$ such that:
	\begin{itemize}
		\item for all $t \in [0,1]$, $\phi(t, \bullet) := \Id + \xi(t, \bullet)$ and $\psi(t, \bullet) := \Id + \zeta(t, \bullet)$ are two diffeomorphisms of $\T^d$, which are inverses of each other,
		\item for all $t \in [0,1]$, 
		\begin{equation*}
		\phi(t, \bullet) \pf \Big( (1 + \varphi(t, \bullet))\Leb \Big) = \Leb,
		\end{equation*}
		or equivalently, 
		\begin{equation}
		\label{eq.pf_psi}
		\psi(t, \bullet) \pf \Leb = (1 + \varphi(t, \bullet)) \Leb,
		\end{equation}
		\item we have:
		\begin{equation}
		\label{eq.estim_xi_zeta}
		\Norm(\xi) + \Norm(\zeta) \leq C \Norm(\varphi). 
		\end{equation}
	\end{itemize}
\end{Thm}

We will also need the following distributional representation of the elements of $\mathcal{E}_0'$, that we state here without a proof.
\begin{Lem}
	\label{E'distribution}
	Take $\alpha\in\mathcal{E}_0'$. There is a unique distributional gradient $F(\alpha) \in \nabla \mathcal{D}'((0,1) \times \T^d)$ such that for all $\varphi \in \mathcal{D}((0,1) \times \T^d; \R^d)$,
	\[
	\cg \alpha,  \Div \varphi \cd_{\mathcal{E}_0', \mathcal{E}_0} = - \cg F(\alpha), \varphi \cd_{\mathcal{D}', \mathcal{D}}.
	\]
	Moreover, $F: \mathcal{E}_0' \mapsto \nabla \mathcal{D}'$ is a continuous injection. In the following, we simply call $\nabla$ this operator. 
\end{Lem}

Finally, we will need the following behavior of the relative entropy with respect to push-forwards. This is a simple consequence of the additivity property of the logarithm.
\begin{Lem}
	Let $\mathcal{X},\mathcal{Y}$ be two Polish spaces, $P,R \in \P(\mathcal{X})$ be two probability measure on $\mathcal{X}$ and $X: \mathcal{X} \to \mathcal{Y}$ be a Borel map. We have:
	\begin{equation}
	\label{eq.entropie_additive}
	H(P|R) = H\big(X \pf P \big| X \pf R\big) + \E_P\Big[H\Big(P^X \Big| R^X\Big)\Big],
	\end{equation}
	where $P^X$ and $R^X$ stand for the conditional probabilities $P(\bullet | X)$ and $R(\bullet | X)$ respectively. In particular,
	\begin{equation}
	\label{eq.entropie_diminue}
	H(X \pf P | X \pf R) \leq H(P|R) .
	\end{equation}
\end{Lem}

\subsection{Girsanov theorem under finite entropy condition}
The link between $\Bro_\nu$ and $\MBro_\nu$ will be seen as a consequence of the following theorem of Girsanov type. This theorem states that if $P$ has finite entropy with respect to $R^\nu$, then $P$ is entirely characterized by its \emph{forward} and \emph{backward Nelson velocities} in the sense of \cite{nel67}. We refer to \cite{leo12} for the proof of this theorem.
\begin{Thm}
	\label{thm.girsanov}
	Take $\nu >0$ and $P \in \P(C^0([0,1]; \T^d))$ with finite entropy with respect to $R^{\nu}$. Then for almost all $t\in [0,1]$, the forward and backward drift:
	\begin{equation}
	\label{eq.b}
	\overrightarrow{b_t} := \lim_{h \to 0} \E_P\left[ \frac{X_{t+h} - X_t}{h} \Big| X_s, \, s \in [0,t] \right] \quad \mbox{and} \quad \overleftarrow{b_t} := \lim_{h \to 0} \E_P\left[ \frac{X_{t} - X_{t-h}}{h} \Big| X_s, \, s \in [t,1] \right]
	\end{equation}
	exist $P$-almost everywhere, have the following integrability:
	\begin{equation*}
	 \int_0^1  \E_P\left[\left| \overrightarrow{b_t} \right|^2 \right]\D t + \int_0^1  \E_P\left[\left| \overleftarrow{b_t} \right|^2 \right]\D t < + \infty ,
	\end{equation*}
	and satisfy in the sense of It\^o:
	\begin{equation}
	\label{eq.SDE_girsanov} \D X_t = \overrightarrow{b_t} \D t + \nu \D B_t \qquad \mbox{and} \qquad \D X_{1-t} = - \overleftarrow{b_t} \D t + \nu \D B_{1-t},
	\end{equation}
	where $(B_t)_{t \in [0,1]}$ is a standard Brownian motion under $P$.
	
	In addition, the entropy of $P$ with respect to $R^{\nu}$ can be expressed in terms of $(\overrightarrow{b_t})_{t \in [0,1]}$ and $(\overleftarrow{b_t})_{t \in [0,1]}$:
	\begin{equation}
	\label{eq.entropie_girsanov}
	\Hbar_\nu(P)  = \nu H(\rho_0 | \Leb) + \frac{1}{2}  \int_0^1  \E_P\left[\left| \overrightarrow{b_t} \right|^2 \right] \D t = \nu H(\rho_1 | \Leb) + \frac{1}{2} \int_0^1  \E_P\left[ \left| \overleftarrow{b_t} \right|^2 \right]\D t,
	\end{equation}
	where $\rho_0$ and $\rho_1$ stand for $X_0{}\pf P$ and $X_1{}\pf P$ respectively. 
\end{Thm}

It is often convenient to use a version of \eqref{eq.entropie_girsanov} where the arrow of time does not intervene. To do so, we do the half some of the two equalities in \eqref{eq.entropie_girsanov} and use the parallelogram identity, in order to get:
\begin{equation}
\label{eq.entropie_girsanov_symetrique}
\Hbar_\nu(P) = \nu \frac{H(\rho_0 | \Leb) + H( \rho_1 | \Leb)}{2} + \frac{1}{2}\int_0^1 \E_P\left[ \left| \frac{\overrightarrow{b_t} + \overleftarrow{b_t}}{2} \right|^2 + \left| \frac{\overrightarrow{b_t} - \overleftarrow{b_t}}{2} \right|^2 \right]\D t.
\end{equation}

Finally, we define the so-called \emph{current} and \emph{osmotic} velocities by the formulas:
\begin{equation}
\label{eq.def_courant_osmotique}
P \mbox{-p.s.} \qquad c(t,X_t) = \E_P\left[ \frac{\overrightarrow{b_t} + \overleftarrow{b_t}}{2} \bigg| X_t \right] \qquad \mbox{and} \qquad w(t,X_t) = \E_P\left[ \frac{\overrightarrow{b_t} - \overleftarrow{b_t}}{2} \bigg| X_t \right].
\end{equation} 
It is easily proved using \eqref{eq.SDE_girsanov} and the It\^o formula that $(\rho,c)$ is a solution of the continuity equation. Moreover, the following result due to F\"ollmer in \cite[Theorem 3.10]{follmer1986time} in the case of dimension 1 and with a straightforward generalization in higher dimension characterizes entirely $w$.

\begin{Thm}[F\"ollmer 1986]
	\label{thm.follmer}
	Under the assumption of Theorem \ref{thm.girsanov}, defining $\rho_t := X_t{}\pf P$ and identifying $\rho_t$ with its density with respect to $\Leb$, for all $t \in [0,1]$, we have:\footnote{Once again, $\nabla \log \rho_t$, is understood through the formula $\nabla \log \rho_t := 2 \nabla \sqrt{\rho_t} / \sqrt{\rho_t}$, using the fact that $\sqrt\rho \in L^2_t H^1_x$.}
	\begin{equation*}
	P \mbox{-p.s.} \qquad w(t,X_t) = \frac{\nu}{2} \nabla \log \rho(t,X_t).
	\end{equation*}
\end{Thm}

As a consequence of formula \eqref{eq.entropie_girsanov_symetrique}\eqref{eq.def_courant_osmotique}, Theorem \ref{thm.follmer} and Jensen's inequality, we get the following inequality:
\begin{equation}
\label{eq.entropy_BB}
\nu \frac{H(\rho_0 | \Leb) + H( \rho_1 | \Leb)}{2} + \H_{\nu}(\rho, c) \leq \Hbar_{\nu}(P).
\end{equation}
This inequality turns out to be an equality if and only if for almost all $t$, $P$ almost everywhere:
\begin{equation}
\label{eq.drift_markov}
c(t,X_t) = \frac{\overrightarrow{b_t} + \overleftarrow{b_t}}{2}  \qquad \mbox{and} \qquad w(t,X_t) = \frac{\overrightarrow{b_t} - \overleftarrow{b_t}}{2}, 
\end{equation}
\emph{i.e.} if and only if $\overrightarrow{b_t}$ and $\overleftarrow{b_t}$ only depend on $X_t$, which is true if and only if $P$ is Markov.

\section{Existence of the pressure in the multiphasic model \sf{MBr\"o}}
\label{sec.existence_MBro}
The purpose of this section is to prove Theorem \ref{thm.existence_pressure_MBro}, so we work with the multiphasic problem $\MBro$, defined in Problem~\ref{pb.Mbro}. We fix $(\I, \m)$ a probability space of labels for the different phases and $\nu>0$ a level of noise. We also fix $\rhorho_0 = (\rho^i_0)_{i \in \I}$ and $\rhorho_1 = (\rho^i_1)_{i \in \I}$ two measurable families of probability measures on $\T^d$ satisfying the incompressibility condition~\eqref{eq.incompressibility_t01} and condition~\eqref{eq.CNS_existence_multiphase}, so that the problem $\MBro_\nu(\rhorho_0, \rhorho_1)$ admits a unique solution.

We first introduce a new optimization problem, relaxing the incompressibility constraint in $\MBro$.

\subsection{A modified optimization problem}
We define this problem as follows.
\begin{Pb}
	\label{pb.Mbro_density}
	Given $M = (M_t)_{t \in [0,1]} \in C^0_0([0,1]; \P(\T^d))$,
	find $(\rhorho, \cc)=(\rho^i, c^i)_{i\in \I}$ a measurable family of solutions of the continuity equation with finite Fischer information, well defined for $\m$-almost every $i$, such that:
	\begin{enumerate}[label=(\alph*)]
		\item \label{item.Mcompatibilitydensity} For $\m$-almost all $i \in \I$:
		\begin{equation*}
		\rho^i|_{t=0} = \rho^i_0 \qquad \mbox{and}\qquad \rho^i|_{t=1} = \rho^i_1.
		\end{equation*}
		\item \label{item.Mincompressibility_density} For all $t \in [0,1]$, we have:
		\begin{equation*}
		\int \rho^i(t,\bullet) \D \m(i) = M_t.
		\end{equation*}
		\item The functional $\HH_\nu(\rhorho, \cc)$ is finite and minimal among the measurable families of distributional solutions of the continuity equation satisfying \ref{item.Mcompatibilitydensity} and \ref{item.Mincompressibility_density}.
	\end{enumerate}
\end{Pb}

From now on, as $\rhorho_0$ and $\rhorho_1$ are supposed to be fixed, we will simply call this problem $\MBro_\nu(M)$, and $\HH_\nu^*(M)$ the corresponding optimal value of $\HH_\nu$. We fix by convention $\HH_\nu^*(M) = + \infty$, if $\MBro_\nu( M)$ has no solution. As $\rhorho_0$ and $\rhorho_1$ satisfy the condition of existence \eqref{eq.CNS_existence_multiphase} for the problem $\MBro_\nu(\rhorho_0, \rhorho_1)$, we know that\footnote{With an abuse of notation, we write $\Leb$ for the curve $t\in [0,1] \mapsto \Leb$.} $\HH_\nu^*(\Leb) < + \infty$.

The functional $\HH_\nu^*( M)$ is convex and lower semi-continuous, as stated in the following lemma.
\begin{Lem}
	\label{lem.Hnustar_convex}
	The functional
	\begin{equation*}
	M \in C^0_0([0,1]; \P(T^d)) \mapsto \HH_\nu^*( M)
	\end{equation*}
	is convex and lower semi-continuous for the topology of $C^0([0,1]; \P(T^d))$.
	
	In particular, it is also semi-continuous for any stronger topology, as the one of $\mathcal{E}$, so that:\footnote{We set $\HH_\nu^*( 1 + \varphi ) = + \infty$ in case $1 + \varphi$ is not everywhere nonnegative.}
	\begin{equation*}
	\varphi \in \mathcal{E}_0 \mapsto \HH_\nu^*( 1 + \varphi )
	\end{equation*}
	is also convex and lower semi-continuous.
\end{Lem}
\begin{proof}
	We start by proving the convexity. Let us take $M_1,M_2 \in C^0_0([0,1]; \P(T^d))$ and $\lambda \in [0,1]$. If $\HH_\nu^*(M_1)$ or $\HH_\nu^*(M_2)$ is infinite, there is nothing to prove. Else, let us consider $(\rhorho^1, \cc^1) = (\rho^{1,i}, c^{1,i})_{i \in \I}$ and $(\rhorho^2, \cc^2)=(\rho^{2,i}, c^{2,i})_{i \in \I}$ the solutions of $\MBro_\nu(M_1)$ and $\MBro_\nu(M_2)$ respectively. We define for $\m$-almost all $i \in \I$ and all $t \in [0,1]$:
	\begin{equation*}
	\widetilde{\rho}{}^{i}_t := (1-\lambda) \rho^{1,i}_t + \lambda \rho^{2,i}_t, \quad \widetilde{m}{}^i_t := (1-\lambda) \rho^{1,i}_t c^{1,i}_t + \lambda \rho^{2,i}_t c^{2,i}_t, \quad \widetilde{c}{}^i_t := \frac{\D \widetilde{m}{}^i_t}{\D \widetilde{\rho}{}^{i}_t},
	\end{equation*}
	and $(\widetilde{\rhorho},\widetilde{\cc}, \widetilde{\boldsymbol{m}}) := (\widetilde{\rho}{}^{i},\widetilde{m}{}^i,\widetilde{c}{}^i)_{i \in \I}$. It is well-known that $\AA$ (see \cite[Proposition 3.4]{bre97}) and $\Fbf$ (straightforward computation), and hence $\HH_\nu$ are convex, when considered as a function of the couple density/momentum, namely $(\widetilde{\rhorho}, \widetilde{\boldsymbol{m}})$ here.
	
	Consequently, as $(\widetilde{\rhorho}, \widetilde{\cc})$ is a competitor for $\MBro_\nu((1-\lambda) M_1 + \lambda M_2)$, we have:
	\begin{align*}
	\HH_\nu^*((1-\lambda) M_1 + \lambda M_2) &\leq \HH_\nu(\widetilde{\rhorho}, \widetilde{\cc})\\ &\leq (1-\lambda) \HH_\nu(\rhorho^1, \cc^1) + \lambda \HH_\nu(\rhorho^2, \cc^2)\\ &= (1-\lambda) \HH_\nu^*(M_1) + \lambda \HH_\nu^*(M_2).
	\end{align*}
	
	The semi-continuity works the same way. We first remark that $\AA$ (still thanks to \cite[Proposition 3.4]{bre97}), $\Fbf$ (by standard arguments), and hence $\HH_\nu$ are lower semi-continuous (when considered as a function of the couple density/momentum). Then, let us take $(M_n)_{n \in \N}$ a sequence of $C^0_0([0,1]; \P(T^d))$ converging to $M$, and $(\rhorho^n, \cc^n)$ the solution of $\MBro_\nu(M_n)$. If $\liminf_n \HH_\nu(\rhorho^n, \cc^n) = + \infty$, there is nothing to prove. Else, up to forgetting some labels, $(\HH_\nu(\rhorho^n, \cc^n))$ is bounded. But then, as $\HH_\nu \geq \AA$, the corresponding sequence $(\rhorho^n, \boldsymbol{m}^n) := (\rho^{n,i}, \rho^{n,i} c^{n,i})_{i \in \I, n \in \N}$ has its values in a compact\footnote{For the convergence in law and almost sure of the measurable map $i \mapsto  (\rho^{n,i}, \rho^{n,i} c^{n,i})_{i \in \I}$, with values in the set of space time measures endowed with the topology of narrow convergence. We are not more specific and refer to \cite{bre97,bre99} for more details.}. If $(\rhorho, \boldsymbol{m})$ is a limit point, and if $(\rhorho, \cc)$ is the corresponding couple of densities/velocities, then it is a competitor for $\MBro_\nu(M)$, and we have:
	\begin{equation*}
	\HH_\nu^*(M) \leq \HH_\nu(\rhorho, \cc) \leq \liminf_{n \to + \infty} \HH_\nu(\rhorho^n, \cc^n) = \liminf_{n \to + \infty} \HH_\nu^*(\rhorho^n, \cc^n). 
	\end{equation*}
	This concludes the proof.
\end{proof}
From now on, we decompose the proof of Theorem \ref{thm.existence_pressure_MBro} into two parts: in Lemma \ref{lem.Hnu_bounded}, we show that $\HH_\nu^*$ is bounded in an $\mathcal{E}_0$-neighborhood of $\Leb$, so that it admits a non-empty subdifferential at $\Leb$. In Lemma~\ref{lem.charact_p}, we show that this subdifferential is a singleton, and we derive formula~\eqref{eq.formule_p_MBro} for $p$, its only element. We conclude the proof of the theorem in Subsection \ref{subsec.conclusion_existence_pressure_MBro}.
\subsection{Boundedness of the optimal value}
Because of Lemma \ref{lem.Hnustar_convex}, and because $\HH_\nu^*(\Leb)< + \infty$, a sufficient condition for $\HH^*_\nu$ to admit a non-empty subdifferential $\partial \HH^*_\nu(\Leb) \subset \mathcal{E}_0'$ at $M = \Leb$ is to be bounded in a $\mathcal{E}_0$-neighbourhood of $\Leb$. This is the subject of the following lemma, which is the main part in the proof of Theorem~\ref{thm.existence_pressure_MBro}. We recall that the norm $\Norm$ is defined by formula \eqref{eq.def_N_Chap5}.

\begin{Lem}
	\label{lem.Hnu_bounded}
	There is $C>0$ only depending on the dimension, $\nu$, $\rhorho_0$ and $\rhorho_1$ such that for all $\varphi \in \mathcal{E}_0$ satisfying the estimate $\Norm(\varphi) \leq 1/2$, we have:
	\begin{equation*}
	\HH_\nu^*(1 + \varphi) \leq C.
	\end{equation*}
\end{Lem}
\begin{proof}
	In the whole proof, the symbol $\lesssim$ means "lower than, up to a multiplicative dimensional constant".
	
	For a given $\varphi$ as in the statement of the lemma, we take $\xi$, $\zeta $ the two vector fields and $\phi$, $\psi$ the corresponding time-dependent diffeomorphisms given by Theorem \ref{dacmos_chap5}. We call $(\rhorho, \cc) = (\rho^i, c^i)_{i \in \I}$ the solution of $\MBro_\nu(\rhorho_0, \rhorho_1)$.
	
	We start by defining a competitor for the problem $\MBro_\nu(1 + \varphi)$.
	
	\paragraph{\underline{Step 1}: Definition of a competitor for $\boldsymbol{\MBro_\nu(1+\varphi)}$.}
	We define for $\m$-almost all $i \in \I$, all $t \in [0,1]$ and $x \in \T^d$:
	\begin{gather}
	\label{eq.def_rhophi} \rho^{\varphi,i}(t,x) := \rho^i(t, \phi(t,x)) \det \Diff \phi(t,x),\\
	\label{eq.def_cphi} c^{\varphi,i}(t,x) := \partial_t \psi(t, \phi(t,x)) + \Diff \psi(t, \phi(t,x)) \cdot c^i(t, \phi(t,x)),
	\end{gather}
	where $\Diff \phi$ and $\Diff \psi$ stand for the differentials with respect to $x$ of $\phi$ and $\psi$ respectively. Then, we call $(\rhorho^{\varphi}, \cc^{\varphi}) := (\rho^{\varphi,i}, c^{\varphi, i})_{i \in \I}$. 
	
	Let us prove that $(\rhorho^{\varphi}, \cc^{\varphi})$ is a competitor for the problem $\MBro_\nu(1 + \varphi)$.
	
	First of all, for $\m$-almost all $i$, $(\rho^{\varphi,i}, c^{\varphi, i})$ is a solution to the continuity equation. Indeed, by the change of variable formula, equation \eqref{eq.def_rhophi} exactly means that for all $t \in [0,1]$ and $\m$-almost all $i$:
	\begin{equation}
	\label{eq.rhophi_pf}
	\rho^{\varphi,i}(t,\bullet)  = \psi(t, \bullet) \pf \rho^i(t, \bullet).
	\end{equation}
	Hence, if $f$ is a test function, we have:
	\begin{align*}
	\frac{\D \phantom{t}}{\D t} \int f(x) \rho^{\varphi,i}(t,x) \D x&= \frac{\D \phantom{t}}{\D t}  \int f( \psi(t,x)) \rho^{i}(t,x)\D x&&\mbox{by \eqref{eq.rhophi_pf}},\\
	&= \int \left\{ \partial_t \Big( f( \psi(t,x)) \Big) +  \Big\cg c^i(t,x) , \nabla \Big( f( \psi(t,x)) \Big) \Big\cd \right\} \rho^i(t,x)\D x\\
	&= \int \Big\cg \nabla  f(\psi(t,x)), \partial_t \psi(t,x) +  \Diff \psi(t,x)\cdot c^i(t,x) \Big\cd  \rho^i(t,x)\D x \\
	&= \int \Big\cg \nabla  f(\psi(t,x)),  c^{\varphi,i}(t,\psi(t,x)) \Big\cd \rho^i(t,x) \D x &&\mbox{by \eqref{eq.def_cphi}}, \\
	&= \int \Big\cg \nabla  f(x),  c^{\varphi,i}(t,x)) \Big\cd  \rho^{\varphi,i}(t,x)\D x  &&\mbox{by \eqref{eq.rhophi_pf}},
	\end{align*}
	where in the second line, we used the fact that $(\rho^i, c^i)$ is a solution to the continuity equation. Hence, the claim.
	
	Moreover, formula \eqref{eq.pf_psi} implies that for all $t \in [0,1]$, the mean density of $\rhorho$ at time $t$ is $1 + \varphi(t, \bullet)$:
	\begin{equation*}
	\int \rho^{\varphi,i}(t,\bullet)  \D \m(i) = \psi(t, \bullet)\pf \int \rho^i(t, \bullet) \D \m(i) = \psi(t, \bullet)\pf \Leb = (1 + \varphi(t, \bullet))\Leb.
	\end{equation*}
	
	As a consequence, $(\rhorho^{\varphi}, \cc^{\varphi})$ is a competitor for $\MBro_\nu(1 + \varphi)$, and:
	\begin{equation}
	\label{eq.rhophi,cphi_competitor}
	\HH_\nu^*(1 + \varphi) \leq \HH_\nu(\rhorho^{\varphi}, \cc^{\varphi}) = \AA(\rhorho^{\varphi}, \cc^{\varphi}) + \nu^2\Fbf(\rhorho^{\varphi}).
	\end{equation}
	(We recall that $\AA$ and $\Fbf$ are defined in \eqref{eq.def_entropy_multiphase}.) To get the result, it remains to estimate $\HH_\nu(\rhorho^{\varphi}, \cc^{\varphi})$. Let us estimate first $\AA(\rhorho^{\varphi}, \cc^{\varphi})$, and then $\Fbf(\rhorho^{\varphi})$.
	\paragraph{\underline{Step 2}: Estimation of $\AA(\rhorho^{\varphi}, \cc^{\varphi})$.} For $i \in \I$, we have:
	\begin{align*}
	\frac{1}{2} \int_0^1 \hspace{-5pt} \int &|c^{\varphi,i}(t,x)|^2 \rho^{\varphi,i}(t,x)\D x \D t  \\
	&= \frac{1}{2} \int_0^1 \hspace{-5pt} \int |c^{\varphi,i}(t,\psi(t,x))|^2  \rho^i(t,x) \D x \D t &&\mbox{by \eqref{eq.rhophi_pf}}, \\
	&= \frac{1}{2} \int_0^1 \hspace{-5pt} \int |\partial_t \psi(t, x) + \Diff \psi(t, x) \cdot c^i(t, x)|^2 \rho^{i}(t,x) \D x \D t &&\mbox{by \eqref{eq.def_cphi}},  \\
	&= \frac{1}{2} \int_0^1 \hspace{-5pt} \int |c^i(t, x) + \partial_t \zeta(t, x)  + \Diff \zeta(t, x) \cdot c^i(t, x)|^2 \rho^i(t,x)\D x \D t &&\mbox{because }\psi = \Id + \zeta,\\
	&\lesssim \frac{1}{2}  \int_0^1 \hspace{-5pt} \int \Big\{ |c^i(t, x)|^2 + |\partial_t \zeta(t, x)|^2 + |\Diff \zeta(t, x) \cdot c^i(t, x)|^2  \Big\} \rho^i(t,x) \D x\D t \\
	&\lesssim \left(1 +  \frac{1}{2}\int_0^1 \hspace{-5pt} \int |c^i(t, x)|^2 \rho^i(t,x) \D x \D t \right) \left(1 + \Norm(\zeta)^2\right),
	\end{align*}
	where the last line is obtained thanks to the definition of $\Norm$ \eqref{eq.def_N_Chap5} and by observing that $\Norm(\zeta)$ controls $\sup_t \Lip(\zeta(t, \bullet))$. It remains to integrate this inequality with respect to $\m$ to obtain:
	\begin{equation}
	\label{eq.estim_AA}
	\AA(\rhorho^{\varphi}, \cc^{\varphi}) \lesssim (1 + \AA(\rhorho, \cc)) \left(1 + \Norm(\zeta)^2\right).
	\end{equation}
	\paragraph{\underline{Step 3}: Estimation of $\Fbf(\rhorho^{\varphi})$.} If $i \in \I$, using the definition \eqref{eq.def_rhophi} of $\rho^{\varphi,i}$, we can compute explicitly for $\m$-almost all $i$ and all $t$ and $x$:
	\begin{equation*}
	\nabla \log \rho^{\varphi,i}(t,x) = {}^t \Diff \phi(t,x) \cdot\nabla \log \rho^i(t, \phi(t,x)) + \nabla \log \det \Diff \phi(t,x),
	\end{equation*}
	where ${}^t \Diff \phi$ is the adjoint of $\Diff \phi$. As a consequence,
	\begin{align*}
	\frac{1}{2} \int_0^1 \hspace{-5pt} \int& \left| \frac{1}{2} \nabla \log \rho^{\varphi,i}(t,x)\right|^2 \rho^{\varphi, i}(t,x) \D x \D t \\
	&= \frac{1}{8} \int_0^1 \hspace{-5pt} \int |\nabla \log \rho^{\varphi,i}(t,\psi(t,x))|^2 \rho^{i}(t,x) \D x \D t &&\mbox{by \eqref{eq.rhophi_pf}},\\
	&= \frac{1}{8} \int_0^1 \hspace{-5pt} \int |{}^t \Diff \phi(t,\psi(t,x)) \cdot\nabla \log \rho^i(t,x) + \nabla \log \det \Diff \phi(t,\psi(t,x))|^2  \rho^{i}(t,x)\D x \D t \\
	&\lesssim \frac{1}{8} \int_0^1 \hspace{-5pt} \int \Big\{ |{}^t \Diff \phi(t,\psi(t,x)) \cdot\nabla \log \rho^i(t,x)|^2 + |\nabla \log \det \Diff \phi(t,\psi(t,x))|^2\Big\} \rho^{i}(t,x) \D x \D t.
	\end{align*}
	The first term can be estimated thanks to:
	\begin{align}
	\notag	\frac{1}{8} \int_0^1 \hspace{-5pt} \int & |{}^t \Diff \phi(t,\psi(t,x)) \cdot\nabla \log \rho^i(t,x)|^2  \rho^{i}(t,x) \D x \D t \\
	\notag	&= \frac{1}{8} \int_0^1 \hspace{-5pt} \int  |(\Id + {}^t \Diff \xi(t,\psi(t,x))) \cdot\nabla \log \rho^i(t,x)|^2  \rho^{i}(t,x) \D x \D t &&\mbox{because }\phi = \Id + \xi,\\
	\label{eq.estim_term1_fischer} &\lesssim \left(\frac{1}{8} \int_0^1 \hspace{-5pt} \int |\nabla \log \rho^{i}(t,x)|^2 \rho^{ i}(t,x) \D x \D t\right) \left(1 + \Norm(\xi)^2\right).
	\end{align}
	For the second term, quick computations show that for all $(t,x)$ where $\xi$ (and consequently $\phi$) is twice differentiable with respect to space:
	\begin{equation}
	\label{eq.nabla_log_det}
	\nabla \log \det \Diff \phi(t,\psi(t,x)) = (\Id + {}^t \Diff \zeta(t,x) )\cdot \nabla \Div \xi (t, \psi(t,x)) ,
	\end{equation} 
	so that:
	\begin{equation*}
	\Big\| 	\nabla \log \det \Diff \phi(t,\psi(t,x)) \Big\|_{\infty} \lesssim (1 + \Norm(\zeta))\Norm(\xi).
	\end{equation*}
	Consequently, we get:
	\begin{equation}
	\label{eq.estim_term2_fischer}
	\frac{1}{8} \int_0^1 \hspace{-5pt} \int |\nabla \log \det \Diff \phi(t,\psi(t,x))|^2  \rho^{i}(t,x) \D x \D t \lesssim  (1 + \Norm(\zeta))^2\Norm(\xi)^2.
	\end{equation}
	Gathering \eqref{eq.estim_term1_fischer} and \eqref{eq.estim_term2_fischer} and integrating with respect to $\m$, we end up with:
	\begin{equation}
	\label{eq.estim_fischer}\Fbf(\rhorho^{\varphi}) \lesssim (1 + \Fbf(\rhorho))\left( 1 + \Norm(\xi)^2 \right)\left( 1 + \Norm(\zeta)^2\right).
	\end{equation}
	\paragraph{\underline{Step 4}: Conclusion.} Gathering the two estimates \eqref{eq.estim_AA} and \eqref{eq.estim_fischer}, inequality \eqref{eq.rhophi,cphi_competitor}, the control \eqref{eq.estim_xi_zeta} on $\xi$ and $\zeta$ given by Theorem \ref{dacmos_chap5} and $\Norm(\varphi) \leq 1/2$, we get:
	\begin{equation*}
	\HH_\nu^*(1 + \varphi) \lesssim 1 + \nu^2 + \HH_\nu(\rhorho, \cc) \lesssim 1 + \nu^2 + \HH_\nu^*(\Leb).
	\end{equation*}	
	The result follows.
\end{proof}
\subsection{Characterization of the pressure as a distribution}
In the following lemma, we show that if $\HH_\nu^*$ admits a non-empty differential at $\Leb$, then its subdifferential is a singleton.
\begin{Lem}
	\label{lem.charact_p}
	Take $p \in \partial \HH_\nu^*(\Leb) \subset \mathcal{E}_0'$. Let $\nabla p$ be the distribution given by Lemma \ref{E'distribution}, and $(\rhorho, \cc) = (\rho^i, c^i)_{i \in \I}$ be the solution of $\MBro_\nu(\rhorho_0, \rhorho_1)$. Then in the sense of distributions:
	\begin{equation}
	\label{eq.pression_MBro}
	- \nabla p = \partial_t \left( \int \rho^i c^i \D \m(i)\right) + \bDiv \left( \int \Big\{ c^i \otimes c^i -  w^i\otimes w^i \Big\} \rho^i \D \m(i) \right), 
	\end{equation}
	where for $i \in \I$:
	\begin{equation*}
	w^i := \frac{\nu}{2} \nabla \log \rho^i.
	\end{equation*}
\end{Lem} 
\begin{proof}
	Take $\xi \in \mathcal{D}((0,1) \times \T^d)$ be a smooth vector field, and define for all $\eps>0$, $t\in [0,1]$ and $x \in \T^d$:
	\begin{gather}
	\label{eq.def_phieps}\phi^\eps(t,x) := x + \eps \xi(t,x),\\
	\notag \varphi^\eps(t,x) := \det \Diff \phi^{\eps}(t,x) - 1.
	\end{gather}
	For all $\eps>0$, the function $\varphi^{\eps}$ belongs to $\mathcal{E}_0$, so that using $p \in \partial \HH_\nu^*(\Leb)$:
	\begin{equation}
	\label{eq.ineg_convexite_phieps}
	\HH_\nu^*(\Leb) + \big\cg p, \varphi^\eps \big\cd_{\mathcal{E}_0', \mathcal{E}_0} \leq \HH_\nu^*(1 + \varphi^\eps).
	\end{equation}
	
	First, we can check that for all $t$ and $x$:
	\begin{equation*}
	\varphi^{\eps}(t,x) = \eps \Div \xi(t,x) + \eps \delta^{\eps}(t,x), 
	\end{equation*}
	where $\delta^{\eps} = \delta^\eps(t,x) \in \R$ tends to zero in any reasonable space of functions. As a consequence, with the notations of Lemma \ref{E'distribution}, we can estimate $\cg p, \varphi^\eps \cd$ in formula \eqref{eq.ineg_convexite_phieps} by:
	\begin{equation}
	\label{eq.estim_bracket}
	\big\cg p, \varphi^\eps \big\cd_{\mathcal{E}_0', \mathcal{E}_0} = - \eps \big\cg \nabla p , \Div \xi \big\cd_{\mathcal{D}', \mathcal{D}} + \! \underset{\eps \to 0}{o}(\eps).
	\end{equation}
	
	It remains to give an estimate for $\HH_\nu^*(1 + \varphi^\eps)$. To do so, we build a competitor for $\MBro_\nu(1 + \varphi^\eps)$ as in the proof of Lemma \ref{lem.Hnu_bounded}, by defining:
	\begin{gather*}
	\rho^{\eps,i}(t,x) := \rho^i(t, \phi^\eps(t,x)) \det \Diff \phi^\eps(t,x),\\
	c^{\eps,i}(t,x) := \partial_t \psi^\eps(t, \phi^\eps(t,x)) + \Diff \psi^\eps(t, \phi^\eps(t,x)) \cdot c^i(t, \phi^\eps(t,x)),
	\end{gather*}
	where $\psi^\eps$ is the spatial inverse of $\phi^\eps$. It is well defined provided $\eps$ is sufficiently small, and it satisfies for all $t$ and $x$:
	\begin{equation}
	\label{eq.DL_psieps}
	\psi^{\eps} (t,x) = x - \eps \xi(t,x) + \eps r^{\eps}(t,x),
	\end{equation}
	where $r^{\eps} = r^\eps(t,x) \in \R^d$ tends to zero in any reasonable space of functions. For all $t$ and $x$, we also call:
	\begin{equation}
	\label{eq.def_zetaeps}
	\zeta^{\eps} (t,x) =\eps \xi(t,x) + \eps r^{\eps}(t,x),
	\end{equation}
	
	As in the proof of Lemma \ref{lem.Hnu_bounded}, $(\rhorho^\eps, \cc^\eps) := (\rho^{\eps,i}, c^{\eps,i})_{i \in \I}$ is a competitor for $\MBro_\nu(1 + \varphi^{\eps})$, so that:
	\begin{equation}
	\label{eq.Hnueps}
	\HH_\nu^*(1 + \varphi^{\eps}) \leq \HH_\nu(\rhorho^\eps, \cc^\eps) = \AA(\rhorho^\eps, \cc^\eps) + \nu^2 \Fbf(\rhorho^\eps).
	\end{equation}
	Once again, we will estimate $\AA(\rhorho^\eps, \cc^\eps)$ and $\Fbf(\rhorho^\eps)$ separately.
	
	\paragraph{Estimation of $\AA(\rhorho^\eps, \cc^\eps)$.} With the same computations as in Step 2 of the proof of Lemma \ref{lem.Hnu_bounded}, we get for $i \in \I$:
	\begin{align*}
	\frac{1}{2} \int_0^1 \hspace{-5pt} \int |c^{\eps,i}(t,x)|^2 \rho^{\eps,i}(t,x) \D x \D t &= \frac{1}{2} \int_0^1 \hspace{-5pt} \int |c^i(t, x) + \partial_t \zeta^\eps(t, x)  + \Diff \zeta^\eps(t, x) \cdot c^i(t, x)|^2 \rho^{i}(t,x) \D x \D t \\
	&= \frac{1}{2} \int_0^1 \hspace{-5pt} \int | c^i(t, x) -\eps \partial_t \xi(t,x) - \eps \Diff \xi(t, x) \cdot c^i(t, x)|^2 \rho^{i}(t,x) \D x \D t+ \!\underset{\eps \to 0}{o}(\eps),
	\end{align*}
	where the second line is obtained using \eqref{eq.def_zetaeps}. By expanding the square, we get:
	\begin{multline*}
	\frac{1}{2} \int_0^1 \hspace{-5pt} \int |c^{\eps,i}(t,x)|^2 \rho^{\eps,i}(t,x) \D x \D t \\ = \frac{1}{2} \int_0^1 \hspace{-5pt} \int |c^{i}(t,x)|^2 \rho^{i}(t,x) \D x \D t - \eps \int_0^1 \hspace{-5pt} \int\Big\cg c^{i}(t,x), \partial_t \xi(t,x) + \Diff \xi(t, x) \cdot c^i(t, x) \Big\cd \rho^{i}(t,x) \D x \D t + \!\underset{\eps \to 0}{o}(\eps).
	\end{multline*}
	Our first estimate is obtained by integrating this inequality with respect to $\m$, and by performing integrations by parts:
	\begin{align}
	\notag \AA(\rhorho^\eps, \cc^\eps) &= \AA(\rhorho, \cc) - \eps \int \hspace{-5pt}\int_0^1 \hspace{-5pt} \int\Big\cg c^{i}(t,x), \partial_t \xi(t,x) + \Diff \xi(t, x) \cdot c^i(t, x) \Big\cd \rho^{i}(t,x) \D x \D t \D \m(i)+ \!\underset{\eps \to 0}{o}(\eps) \\
	\label{eq.estim_Aeps} &=  \AA(\rhorho, \cc) + \eps \left\cg \partial_t \left( \int c^i \rho^i \D \m(i) \right) + \bDiv \left( \int c^i \otimes c^i \rho^i \D \m(i) \right), \xi \right\cd_{\mathcal{D}', \mathcal{D}} + \!\underset{\eps \to 0}{o}(\eps).
	\end{align}
	\paragraph{Estimation of $\Fbf(\rhorho^\eps)$.} Here, for $i \in \I$, the computations of Step 3 of the proof of Lemma \ref{lem.Hnu_bounded} give:
	\begin{multline*}
	\frac{1}{8} \int_0^1 \hspace{-5pt} \int |\nabla \log \rho^{\eps,i}(t,x)|^2 \rho^{\eps, i}(t,x) \D x \D t\\
	= \frac{1}{8} \int_0^1 \hspace{-5pt} \int |{}^t \Diff \phi^\eps(t,\psi^\eps(t,x)) \cdot\nabla \log \rho^i(t,x) + \nabla \log \det \Diff \phi^\eps(t,\psi^\eps(t,x))|^2 \rho^{i}(t,x) \D x \D t.
	\end{multline*}
	But because of \eqref{eq.def_phieps}, \eqref{eq.DL_psieps} and \eqref{eq.nabla_log_det}, we have for all $t$ and $x$:
	\begin{gather*}
	{}^t \Diff \phi^\eps(t,\psi^\eps(t,x)) = \Id + \eps \, {}^t \Diff \xi(t,x) + \!\underset{\eps \to 0}{o}(\eps),\\
	\nabla \log \det \Diff \phi^\eps(t,\psi^\eps(t,x)) = \eps \nabla \Div \xi(t,x) + \!\underset{\eps \to 0}{o}(\eps), 
	\end{gather*}
	where the $o(\eps)$ is uniform in $t$ and $x$. Plugging these equalities in the previous formula leads to:
	\begin{align*}
	\frac{1}{8} \int_0^1 \hspace{-5pt} \int& |\nabla \log \rho^{\eps,i}(t,x)|^2 \rho^{\eps, i}(t,x) \D x \D t\\
	&= \frac{1}{8} \int_0^1 \hspace{-5pt} \int |( \Id + \eps \, {}^t \Diff \xi(t,x)) \cdot\nabla \log \rho^i(t,x) +  \eps \nabla \Div \xi(t,x) |^2 \rho^{i}(t,x) \D x \D t + \!\underset{\eps \to 0}{o}(\eps) \\
	&= 	\frac{1}{8} \int_0^1 \hspace{-5pt} \int |\nabla \log \rho^{i}(t,x)|^2 \rho^{i}(t,x) \D x \D t \\
	&\qquad + \frac{\eps}{4} \int_0^1 \hspace{-5pt} \int \Big\cg \nabla \log \rho^{i}(t,x), \Diff \xi(t,x) \cdot \nabla \log \rho^{i}(t,x) + \nabla \Div \xi(t,x) \Big\cd \rho^{i}(t,x) \D x \D t + \!\underset{\eps \to 0}{o}(\eps).
	\end{align*}
	Integrating with respect to $\m$, multiplying by $\nu^2$, calling $w^i :=\nu \nabla \log \rho^i/2$ and performing integrations by parts, we get:
	\begin{align}
	\notag &\qquad\nu^2 \Fbf(\rhorho^\eps) \\
	\notag &= \nu^2 \Fbf(\rhorho) + \frac{\eps}{4} \nu^2 \int \hspace{-5pt}\int_0^1 \hspace{-5pt} \int \Big\cg \nabla \log \rho^{i}(t,x) , \Diff \xi(t,x) \cdot \nabla \log \rho^{i}(t,x) + \nabla \Div \xi(t,x) \Big\cd \rho^{i}(t,x) \D x \D t \D \m(i) + \!\underset{\eps \to 0}{o}(\eps)\\
	\notag &=\nu^2 \Fbf(\rhorho) - \eps  \left\cg \bDiv \left( \int w^i \otimes w^i \rho^i \D \m(i) \right) - \nabla \Div\left(\frac{\nu}{2} \int w^i \rho^i \D \m(i)\right), \xi \right\cd_{\mathcal{D}', \mathcal{D}} + \!\underset{\eps \to 0}{o}(\eps)\\
	\label{eq.estim_Feps}& = \nu^2 \Fbf(\rhorho) - \eps  \left\cg \bDiv \left( \int w^i \otimes w^i \rho^i \D \m(i) \right), \xi \right\cd_{\mathcal{D}', \mathcal{D}} + \!\underset{\eps \to 0}{o}(\eps),
	\end{align}
	where the last line is obtained using:
	\begin{equation*}
	\int w^i \rho^i \D \m(i) = \frac{1}{2} \int \nabla \rho^i \D \m(i) = \frac{1}{2}\nabla \int \rho^i \D \m(i) = \frac{1}{2} \nabla \Leb = 0. 
	\end{equation*}
	
	\paragraph{Conclusion.} Hence, gathering the convex inequality \eqref{eq.ineg_convexite_phieps}, the expansion of the bracket \eqref{eq.estim_bracket}, inequality \eqref{eq.Hnueps} and the two estimates \eqref{eq.estim_Aeps} and \eqref{eq.estim_Feps}, we get:
	\begin{multline*}
	\HH_\nu^*(\rhorho, \cc) - \eps \big\cg \nabla p , \Div \xi \big\cd_{\mathcal{D}', \mathcal{D}} \\
	\leq \HH_\nu^*(\rhorho, \cc) + \eps  \left\cg \partial_t \left( \int c^i \rho^i \D \m(i) \right) + \bDiv \left( \int \Big\{ c^i \otimes c^i - w^i \otimes w^i \Big\} \rho^i \D \m(i) \right), \xi \right\cd_{\mathcal{D}', \mathcal{D}} + \!\underset{\eps \to 0}{o}(\eps).
	\end{multline*}
	Letting $\eps$ go to zero, this formula implies that for all $\xi \in \mathcal{D}((0,1) \times \T^d)$,
	\begin{equation*}
	\left\cg \partial_t \left( \int c^i \rho^i \D \m(i) \right) + \bDiv \left( \int \Big\{ c^i \otimes c^i - w^i \otimes w^i \Big\} \rho^i \D \m(i) \right) + \nabla p, \xi \right\cd_{\mathcal{D}', \mathcal{D}} \geq 0.
	\end{equation*}
	But replacing $\xi$ by $- \xi$, this inequality is in fact an equality, and it exactly means that \eqref{eq.pression_MBro} holds in a distributional sense.
\end{proof}
\subsection{Conclusion of the proof of Theorem \ref{thm.existence_pressure_MBro}}
\label{subsec.conclusion_existence_pressure_MBro}

Theorem \ref{thm.existence_pressure_MBro} follows easily from Lemma \ref{lem.Hnustar_convex}, Lemma \ref{lem.Hnu_bounded} and Lemma \ref{lem.charact_p}. Because of Lemma~\ref{lem.Hnustar_convex}, $\HH_\nu^*$ is convex and lower semi-continuous, and thanks to Lemma \ref{lem.charact_p}, we can find $p \in \partial \HH_\nu^* \subset \mathcal{E}_0'$ \emph{i.e.} such that for all $\varphi \in \mathcal{E}_0$,
\begin{equation*}
\HH_\nu^*(1 + \varphi) \geq \HH_\nu^*(\Leb) + \cg p , \varphi \cd.
\end{equation*}
But in that case, if $(\widetilde\rhorho, \widetilde\cc)$ is as in the statement of Theorem \ref{thm.existence_pressure_MBro}, and if $(\rhorho, \cc)$ is the solution of $\MBro_\nu(\rhorho_0, \rhorho_1)$, then
\begin{equation*}
\HH_\nu(\widetilde\rhorho, \widetilde\cc) \geq \HH_\nu^*(1 + \varphi) \geq \HH_\nu^*(\Leb) + \cg p , \varphi \cd = \HH_\nu(\rhorho, \cc) + \cg p , \varphi \cd .
\end{equation*}
Uniqueness and formula \eqref{eq.formule_p_MBro} are directly given by Lemma \ref{lem.charact_p}. \qed

\section{Link between \sf{MBr\"o} and \sf{Br\"o}}
\label{sec.lien_Bro_MBro}
\subsection{Statement of the result}
We take $\gamma$ a bistochastic measure satisfying the condition \eqref{eq.condition_existence_process} of existence of a solution for the problem $\Bro_\nu(\gamma)$, defined in Problem~\ref{pb.Bro}. We call as in the introduction $R^\nu \in \mathcal{P}(C^0([0,1]; \R^d))$ the law of the Brownian motion starting from $\Leb$. Let $P$ be a solution of $\Bro_\nu(\gamma)$. We define for $\gamma$-almost all $(x,y) \in \T^d \times \T^d$:
\begin{equation}
\label{eq.def_Pxy}
P^{x,y} := P(\bullet | X_0 =x \mbox{ and } X_1 = y).
\end{equation}
(As usual, if $t \in [0,1]$, $X_t$ is the evaluation map at time $t$.) Also call $R^{\nu,x,y}$ the Brownian bridge:
\begin{equation}
\label{eq.def_Rnuxy}
R^{\nu, x,y} := R^\nu(\bullet | X_0 =x \mbox{ and } X_1 = y).
\end{equation}

For $a<b \in [0,1]$, we call $X_{[a,b]}$ the restriction operator:
\begin{equation*}
X_{[a,b]}: \omega \in C^0([0,1]; \T^d) \mapsto \omega|_{[a, b]} \in  C^0([a,b]; \T^d).
\end{equation*}
Then, we define:
\begin{equation}
\label{eq.def_Pxyeps}
P^{x,y}_\eps := X_{[\eps,1-\eps]} {}\pf P^{x,y}, \quad R^\nu_\eps :=  X_{[\eps,1-\eps]} {}\pf R^\nu, \quad  R^{\nu,x,y}_\eps :=  X_{[\eps,1-\eps]} {}\pf R^{\nu,x,y}.
\end{equation}
We will prove that the family of couples of density and current velocity of $P^{x,y}$ (in the sense of formula \eqref{eq.def_courant_osmotique} and Theorem \ref{thm.girsanov} of the introduction), with $(x,y) \in \T^d \times \T^d$ is a solution of the problem $\MBro_\nu$ with respect to its own endpoints, with $\I = \T^d \times \T^d$ and $\m = \gamma$.

\begin{Thm}
	\label{thm.lien_Bro_MBro}
	Take $\gamma$ a bistochastic measure satisfying the condition \eqref{eq.condition_existence_process} of existence for $\Bro_\nu(\gamma)$, $P$ the solution of $\Bro_\nu(\gamma)$, $\eps\in (0,1/2)$ and for $\gamma$-almost all $(x,y) \in \T^d \times \T^d$, consider $P^{x,y}_\eps$ as defined by formula~\eqref{eq.def_Pxyeps}.
	
	For $\gamma$-almost all $(x,y)$, we have:
	\begin{equation}
	\label{eq.finite_entropy}
	H(P^{x,y}_\eps | R^\nu_\eps) < + \infty.
	\end{equation}
	For all $t \in [0, 1]$, call $\rho^{x,y}_{t} := X_t {}\pf P^{x,y}$ and take $c^{x,y}:[\eps, 1-\eps] \times \T^d \to \R^d$ the current velocity of $P^{x,y}_\eps$ as given by formula \eqref{eq.def_courant_osmotique} and Theorem \ref{thm.girsanov}.\footnote{\label{footnote.restriction_time} \emph{A priori}, $c^{x,y}$ depends on $\eps$. In fact, we can show with \eqref{eq.b}\eqref{eq.def_courant_osmotique} that if $\eps_1 < \eps_2$, then $c_{\eps_2}^{x,y}$ is the restriction of $c^{x,y}_{\eps_1}$ to the set of times $[\eps_2, 1-\eps_2]$. For this reason, and to lighten the notations, we will omit dependence of $c^{x,y}$ in $\eps$.}
	
	Then $(\rhorho, \cc) := (\rho^{x,y}, c^{x,y})_{(x,y) \in \T^d \times \T^d}$ is the solution of $\MBro_\nu(\rhorho_{\eps}, \rhorho_{1-\eps})$ between the times $t=\eps$ and $t=1 - \eps$, with $\I = \T^d \times \T^d$ and $\m = \gamma$.
\end{Thm}

To prove this theorem, we will need two lemmas. The first one will be useful to show that for $\gamma$-almost all $(x,y)$, $P^{x,y}_\eps$ has a finite entropy with respect to $R^\nu_\eps$. It writes as follows.
\begin{Lem}
	\label{lem.time_restriction}
	Take $\eps \in (0,1/2)$ and $x,y \in \T^d$. We have $R_\eps^{\nu,x,y} \ll R^\nu_\eps$, and there exist positive smooth functions $f^{\nu,x,y}_\eps$ and $g^{\nu,x,y}_\eps$ on $\T^d$ such that the Radon-Nikodym derivative of $R^{\nu,x,y}_\eps$ with respect to $R^\nu_\eps$ is:
	\begin{equation}
	\label{eq.bridge_sol_schro}
	\frac{\D R^{\nu,x,y}_\eps}{\D R^\nu_\eps\phantom{{}^{,x,y}}} = f^{\nu,x,y}_\eps(X_\eps) g^{\nu,x,y}_\eps(X_{1- \eps}) 
	\end{equation}
	
	As a consequence, for all $Q \in \P(C^0([0,1]; \T^d))$, we have:
	\begin{equation}
	\label{eq.entropy_bridge}
	H(Q_\eps | R^{\nu}_\eps) = H(Q_\eps | R^{\nu, x, y}_\eps) + \int \log f^{\nu,x,y}(x) \D \rho^Q_\eps(x) + \int \log g_\eps^{\nu,x,y}(x) \D \rho^Q_{1-\eps}(x),
	\end{equation}
	where $Q_\eps := X_{[\eps,1-\eps]}{}\pf Q$, $\rho^Q_\eps := X_\eps {}\pf Q$ and $\rho^Q_{1-\eps} := X_{1-\eps} {}\pf Q$.
\end{Lem} 
\begin{Rem}
	The first point of the lemma implies that up to time restrictions, the Brownian bridge $R^{\nu,x,y}$ is the solution of the dynamical Schr\"odinger problem with respect to its own endpoints, see \cite[Theorem~3.3]{leo13}. 
\end{Rem}

The second lemma will let us associate to a solution $(\rho,c)$ to the continuity equation with $\H_\nu(\rho,c) < +\infty$ (recall that $\H_\nu$ is defined by formula \eqref{eq.def_Hnu}) a (Markov) process $Q$ satisfying $\nu H(Q | R^\nu) \leq \H_\nu(\rho,c)\, +$ endpoint terms, and whose density is $\rho$.
\begin{Lem}
	\label{lem.noisy_superposition_principle}
	Let $(\rho,c)$ be a solution to the continuity equation with:
	\begin{equation*}
	\H_{\nu}(\rho,c) < +\infty.
	\end{equation*}
	($\H_\nu$ is defined by \eqref{eq.def_Hnu}.) There exist $Q \in \P(C^0([0,1]; \T^d))$ such that:
	\begin{itemize}
		\item the entropy of $Q$ with respect to $R^\nu$ is given by:\footnote{We could check that our construction leads to a law $Q$ whose current velocity is $c$, and hence because of inequality~\eqref{eq.entropy_BB}, this inequality is in fact an equality. But as we will not need this fact, we will not prove it.}
		\begin{equation}
		\label{eq.entropy_converse_BB}
		\Hbar_\nu(Q) \leq \nu \frac{H(\rho_0| \Leb) + H(\rho_1 | \Leb)}{2} + \H_\nu(\rho,c) < +\infty,
		\end{equation}
		\item for all $t \in [0,1]$, $X_t {}\pf Q = \rho_t$.
	\end{itemize}
\end{Lem}

We prove Theorem \ref{thm.lien_Bro_MBro} in the next subsection and postpone the proof of Lemma \ref{lem.time_restriction} to Subsection \ref{subsec.proof_lemma_time_restriction}, and the proof of Lemma \ref{lem.noisy_superposition_principle} to Subsection \ref{subsec.proof_lemma_noisy_superposition}.

In these proofs, we will have to build laws $P$ on $C^0([0,1]; \T^d)$ by \emph{concatenation}. The idea is the following. Let $a<b<c \in [0,1]$ be three given times, and $P_1,P_2$ be laws on $C^0([a,b]; \T^d)$ and $C^0([b,c]; \T^d)$ respectively. In the case when there is $x \in \T^d$ such that $P_1$-almost everywhere and $P_2$-almost everywhere, $X_b = x$, we will denote by:
\begin{equation*}
P_1 \otimes P_2 \in \P(C^0([a,c]; \T^d))
\end{equation*}
the product measure of $P_1$ and $P_2$ \emph{via} the identification:
\begin{equation*}
C^0([a,c]; \T^d) \cap \{X_b = x\} = \Big( C^0([a,b]; \T^d) \cap \{X_b = x\} \Big) \times \Big( C^0([b,c]; \T^d) \cap \{X_b = x\} \Big).
\end{equation*}
This construction is easily adapted when there are more than two laws.

Also, if $0<a<b<1$, and if $P$ is a law on $C^0([0,1]; \T^d)$, the conditional law:
\begin{equation*}
P\Big(  \bullet \Big| X_{[a,b]}\Big),
\end{equation*} 
which is well defined $P$-almost everywhere, will be seen as an element of:
\begin{equation*}
\P\Big(C^0([0,a]; \T^d) \times C^0([b,1]; \T^d)\Big).
\end{equation*}

\subsection{Proof of Theorem \ref{thm.lien_Bro_MBro} using Lemma \ref{lem.time_restriction} and Lemma \ref{lem.noisy_superposition_principle}}
Take $\gamma$, $P$, $\eps$ as in the statement of the theorem. Let us first check the entropy condition \eqref{eq.finite_entropy}. By the disintegration formula for the entropy \eqref{eq.entropie_additive} used with the map $X := (X_0, X_1)$, we have:
\begin{equation}
\label{eq.disintegration_0,1}
H( P|R^\nu ) = H\big( \gamma \big| (X_0,X_1)\pf R^\nu \big) + \E_P\Big[ H\big(P^{X_0,X_1} \big| R^{\nu, X_0, X_1}\big) \Big],
\end{equation}
where $(P^{x,y})$ and $(R^{\nu,x,y})$ are defined in \eqref{eq.def_Pxy} and \eqref{eq.def_Rnuxy} respectively. In particular, as all the entropies are nonnegative and as $H(P|R^\nu) < +\infty$, we have;
\begin{equation*}
P\mbox{-almost everywhere}, \qquad H\big(P^{X_0,X_1} \big| R^{\nu, X_0, X_1}\big) < +\infty,
\end{equation*}  
which exactly means that $H\big(P^{x,y} \big| R^{\nu, x, y}\big)< +\infty$ for $\gamma$-almost all $(x,y)$. Using formula \eqref{eq.entropie_diminue} with $X = X_{[\eps,1-\eps]}$, the restriction operator, we deduce that:
\begin{equation*}
\mbox{for }\gamma\mbox{-almost all }(x,y) , \qquad H\big(P^{x,y}_\eps \big| R^{\nu, x,y}_\eps \big) < +\infty.
\end{equation*} 
We conclude by estimating $H\big(P^{x,y}_\eps \big| R^{\nu, x,y}_\eps \big)$ with the help of formula \eqref{eq.entropy_bridge} of Lemma \ref{lem.time_restriction}, using the fact that $f^{\nu,x,y}_\eps$ and $g^{\nu,x,y}_\eps$ are bounded away from $0$.

Hence, we consider $(\rhorho, \cc)$ as in the statement of the theorem. We need to prove that this is a solution of $\MBro_\nu(\rhorho_{\eps}, \rhorho_{ 1-\eps})$ between the times $\eps$ and $1 - \eps$. So let us take an other competitor $(\widetilde{\rhorho}, \widetilde{\cc}) = (\widetilde{\rho}{}^{x,y}, \widetilde{c}{}^{x,y})_{(x,y) \in \T^d \times \T^d}$ for this problem (in particular, $(\widetilde{\rhorho}, \widetilde{\cc})$ is only defined between the times $\eps$ and $1-\eps$). We will build from $(\widetilde{\rhorho}, \widetilde{\cc})$ a competitor $Q$ for $\Bro_\nu(\gamma)$.

\paragraph{Construction of $\boldsymbol{Q}$.} For $(x,y) \in \T^d \times \T^d$, the following construction of $Q^{x,y}$ clearly makes sense $\gamma$-almost everywhere.
First we define:
\begin{equation*}
P_{\eps,0}^{x,y} := X_{[0,\eps]} {}\pf P^{x,y}, \qquad \mbox{and} \qquad P_{\eps,1}^{x,y} := X_{[1-\eps,1]} {}\pf P^{x,y}.
\end{equation*}
Then, we take $Q^{x,y}_\eps \in \P(C^0([\eps, 1-\eps]; \T^d))$ as given by Lemma \ref{lem.noisy_superposition_principle} from $(\widetilde{\rho}{}^{x,y}, \widetilde{c}{}^{x,y})$. We define $Q^{x,y}$ by concatenation:
\begin{equation}
\label{eq.def_Qxy}
Q^{x,y} := \int_{C^0([\eps, 1-\eps]; \T^d)} P_{\eps,0}^{x,y}(\bullet | X_\eps = \omega_\eps) \otimes \delta_{\omega} \otimes P_{\eps,1}^{x,y}(\bullet | X_{1-\eps} = \omega_{1-\eps}) \D Q^{x,y}_\eps(\omega).
\end{equation}
Finally, we define $Q$ by:
\begin{equation}
\label{eq.def_Q}
Q := \int Q^{x,y} \D \gamma(x,y).
\end{equation}

\paragraph{Marginal laws of $\boldsymbol{Q^{x,y}}$.} From formula \eqref{eq.def_Qxy}, we easily get:
\begin{align}
\notag X_{[\eps,1-\eps]} {}\pf Q^{x,y} &= \int X_{[\eps,1-\eps]} {}\pf\Big\{ P_{\eps,0}^{x,y}(\bullet | X_\eps = \omega_\eps) \otimes \delta_{\omega} \otimes P_{\eps,1}^{x,y}(\bullet | X_{1-\eps} = \omega_{1-\eps}) \Big\}\D Q^{x,y}_\eps(\omega)\\
\label{eq.marginal_middle}& = \int \delta_{\omega} \D Q^{x,y}_\eps(\omega) = Q^{x,y}_\eps.
\end{align}
In addition, as $X_\eps {}\pf Q^{x,y}_\eps = \rho^{x,y}_\eps = X_\eps {}\pf P_{0,\eps}^{x,y}$, we also have:
\begin{align}
\notag X_{[0,\eps]} {}\pf Q^{x,y} &= \int X_{[0,\eps]} {}\pf\Big\{ P_{\eps,0}^{x,y}(\bullet | X_\eps = \omega_\eps) \otimes \delta_{\omega} \otimes P_{\eps,1}^{x,y}(\bullet | X_{1-\eps} = \omega_{1-\eps}) \Big\}\D Q^{x,y}_\eps(\omega)\\
\notag & = \int  P_{\eps,0}^{x,y}(\bullet | X_\eps = \omega_\eps)  \D Q^{x,y}_\eps(\omega)\\
\label{eq.marginal_beginning}& = \int  P_{\eps,0}^{x,y}(\bullet | X_\eps = z)  \D \rho^{x,y}_\eps |_{t=\eps} (z) = P_{0,\eps}^{x,y}.
\end{align}
In the same way:
\begin{equation}
\label{eq.marginal_end}
X_{[1-\eps,1]} {}\pf Q^{x,y} = P_{1,\eps}^{x,y}
\end{equation}

\paragraph{The law $\boldsymbol{Q}$ is a competitor for Br\"odinger.}
First $Q^{x,y}$-almost all path joins $x$ to $y$, so that by \eqref{eq.def_Q}:
\begin{equation*}
(X_0,X_1)\pf Q = \gamma.
\end{equation*}

Let us check the incompressibility. From formulae \eqref{eq.marginal_middle}, \eqref{eq.marginal_beginning} and \eqref{eq.marginal_end}, we deduce that for all $t \in [0,1]$ and $\gamma$-almost all $(x,y) \in \T^d \times \T^d$:
\begin{equation*}
X_t {} \pf Q^{x,y} = \left\{  
\begin{aligned}
&\rho^{x,y}_t &&\mbox{if }t \in [0,\eps],\\
&\widetilde{\rho}{}^{x,y}_t &&\mbox{if }t \in [\eps,1-\eps],\\
&\rho^{x,y}_t &&\mbox{if }t \in [1-\eps,1].
\end{aligned}
\right.
\end{equation*}
Consequently, if $t \in [0,\eps] \cup [1-\eps, \eps]$, we have:
\begin{equation*}
X_t {}\pf Q = \int \rho^{x,y}_t \D \gamma(x,y) = \int X_t {}\pf P^{x,y} \D \gamma(x,y) = X_t {}\pf \int P^{x,y} \D \gamma(x,y) = X_t {}\pf P = \Leb,
\end{equation*}
because $P$, as the solution of $\Bro_\nu(\gamma)$, is incompressible and compatible with $\gamma$. If $t \in [\eps, 1-\eps]$,
\begin{equation*}
X_t {}\pf Q = \int \widetilde{\rho}^{x,y}_t \D \gamma(x,y) = \Leb,
\end{equation*}
because $(\widetilde{\rhorho}, \widetilde{\cc})$ is a competitor for $\MBro(\rhorho_\eps, \rhorho_{1-\eps})$ with $(\I, \m)=(\T^d \times \T^d, \gamma)$, and is hence incompressible. We conclude that $Q$ is a competitor for $\Bro_\nu(\gamma)$. In particular:
\begin{equation}
\label{eq.P_sol}
\Hbar_{\nu}(Q) = \nu H(Q | R^\nu) \geq \nu H(P | R^\nu) = \Hbar_{\nu}(P).
\end{equation}
From now on, the goal is to express the entropies $H(P | R^\nu)$ and $H(Q | R^\nu)$ in terms of $\HH_\nu(\rhorho, \cc)$ and $\HH_\nu(\widetilde{\rhorho}, \widetilde{\cc})$, and to use \eqref{eq.P_sol} to compare $\HH_\nu(\rhorho, \cc)$ and $\HH_\nu(\widetilde{\rhorho}, \widetilde{\cc})$.\footnote{\label{footnote.redution_time} With a slight abuse of notation, we still call $\H_\nu$ and $\HH_\nu$ the functionals defined by formulas \eqref{eq.def_Hnu}\eqref{eq.def_entropy_multiphase}, but only integrating between the times $\eps$ and $1-\eps$.}

\paragraph{Computation of the entropy of $\boldsymbol{P}$.}
We first compute $H(P | R^\nu)$. First, \eqref{eq.disintegration_0,1} can be rewritten:
\begin{equation}
\label{eq.disintegration_endpoints}
H(P | R^\nu) = H\big(\gamma \big| (X_0,X_1) \pf R^\nu) + \int H(P^{x,y} | R^{\nu,x,y}) \D \gamma(x,y).
\end{equation}
Then to compute $H(P^{x,y} | R^{\nu,x,y})$, we use the additive property of the entropy \eqref{eq.entropie_additive}, but this time with $X = X_{[\eps,1-\eps]}$. This leads for $\gamma$-almost all $(x,y)$ to:
\begin{equation}
\label{eq.disintegration_eps}
H(P^{x,y} | R^{\nu,x,y}) = H(P_{\eps}^{x,y} | R^{\nu,x,y}_\eps) + \E_{P^{x,y}} \Big[ H\Big(P^{x,y}(\bullet | X_{[\eps,1-\eps]}) \Big| R^{\nu,x,y}(\bullet | X_{[\eps,1-\eps]})\Big) \Big]
\end{equation}
We compute the first term thanks to formula \eqref{eq.entropy_bridge} of Lemma \ref{lem.time_restriction} and inequality \eqref{eq.entropy_BB}:
\begin{multline}
\label{eq.entropy_wrt_bridge}
\nu H(P_{\eps}^{x,y} | R^{\nu,x,y}_\eps) \geq \nu \frac{H(\rho_\eps^{x,y} | \Leb) + H( \rho_{1-\eps}^{x,y} | \Leb)}{2} \\- \nu \int \log f^{\nu,x,y}_\eps \D \rho_\eps^{x,y} - \nu \int \log g_\eps^{\nu,x,y}\D \rho_{1-\eps}^{x,y} + \H_\nu(\rho^{x,y}, c^{x,y}).
\end{multline}

\paragraph{An inequality for the second term.} On the other hand, as $R^{\nu,x,y}$ is Markovian. In particular, calling:
\begin{equation*}
R_{\eps,0}^{\nu,x,y} := X_{[0,\eps]} {}\pf R^{\nu,x,y}, \qquad \mbox{and} \qquad R_{\eps,1}^{\nu,x,y} := X_{[1-\eps,1]} {}\pf R^{\nu,x,y},
\end{equation*}
we have:
\begin{gather*}
X_{[0,\eps]} {}\pf R^{\nu,x,y}(\bullet | X_{[\eps,1-\eps]}) = R_{\eps,0}^{\nu,x,y}(\bullet | X_\eps),\\
R^{\nu,x,y}(\bullet | X_{[0,1-\eps]}) = R_{\eps,1}^{\nu,x,y}(\bullet|X_{1-\eps}).
\end{gather*}
Consequently, using \eqref{eq.entropie_additive} with $X = X_{[0, \eps]}$, we have $P^{x,y}$-almost surely:
\begin{align*}
&\quad H\Big(P^{x,y}(\bullet | X_{[\eps,1-\eps]}) \Big| R^{\nu,x,y}(\bullet | X_{[\eps,1-\eps]})\Big) \\
&= H\Big( X_{[0,\eps]} {}\pf P^{x,y}(\bullet | X_{[\eps,1-\eps]}) \Big| R^{\nu,x,y}_{\eps,0}(\bullet | X_\eps) \Big) + \E_{P^{x,y}(\bullet | X_{[\eps, 1-\eps]})}\Big[ H\Big( P^{x,y}(\bullet | X_{[0,1-\eps]}) \Big| R^{\nu,x,y}_{\eps,1}(\bullet|X_{1-\eps})\Big)\Big]\\
&= H\Big( X_{[0,\eps]} {}\pf P^{x,y}(\bullet | X_{[\eps,1-\eps]}) \Big| R^{\nu,x,y}_{\eps,0}(\bullet | X_\eps) \Big) + \E_{P^{x,y}}\Big[ H\Big( P^{x,y}(\bullet | X_{[0,1-\eps]}) \Big| R^{\nu,x,y}_{\eps,1}(\bullet|X_{1-\eps})\Big) \Big| X_{[\eps, 1-\eps]}\Big].
\end{align*}

Remark the following identities:
\begin{gather*}
\E_{P^{x,y}}\Big[ X_{[0,\eps]} {}\pf P^{x,y}(\bullet | X_{[\eps,1-\eps]}) \Big| X_\eps \Big] = P^{x,y}_{\eps,0}(\bullet|X_\eps), \\
\E_{P^{x,y}}\Big[ P^{x,y}(\bullet | X_{[0,1-\eps]}) \Big| X_{1-\eps} \Big] = P^{x,y}_{\eps,1}(\bullet|X_{1-\eps}).
\end{gather*}
Integrating the previous formula with respect to $P^{x,y}$ and using Jensen's inequality in the last line:
\begin{align}
\notag \E_{P^{x,y}} \Big[& H\Big(P^{x,y}(\bullet | X_{[\eps,1-\eps]}) \Big| R^{\nu,x,y}(\bullet | X_{[\eps,1-\eps]})\Big) \Big]\\
\notag &=\E_{P^{x,y}} \Big[  H\Big( X_{[0,\eps]} {}\pf P^{x,y}(\bullet | X_{[\eps,1-\eps]}) \Big| R^{\nu,x,y}_{\eps,0}(\bullet | X_\eps) \Big)  \Big]\\ 
\notag &\hspace{5cm}+ \E_{P^{x,y}} \Big[ H\Big(P^{x,y}(\bullet | X_{[0,1-\eps]}) \Big| R^{\nu,x,y}_{\eps,1}(\bullet|X_{1-\eps})\Big) \Big] \\
\notag &=  \E_{P^{x,y}} \Big[ \E_{P^{x,y}} \Big[  H\Big( X_{[0,\eps]} {}\pf P^{x,y}(\bullet | X_{[\eps,1-\eps]}) \Big| R^{\nu,x,y}_{\eps,0}(\bullet | X_\eps) \Big) \Big| X_\eps \Big] \Big] \\
\notag &\hspace{5cm}+  \E_{P^{x,y}} \Big[ \E_{P^{x,y}} \Big[ H\Big(P^{x,y}(\bullet | X_{[0,1-\eps]}) \Big| R^{\nu,x,y}_{\eps,1}(\bullet|X_{1-\eps})\Big) \Big| X_{1-\eps}\Big]\Big]\\
\label{eq.P_not_Markov_apriori} &\geq \E_{P^{x,y}}  \Big[  H\Big( P^{x,y}_{\eps,0}(\bullet|X_\eps) \Big| R^{\nu,x,y}_{\eps,0}(\bullet | X_\eps) \Big) \Big] +   \E_{P^{x,y}} \Big[ H\Big(P^{x,y}_{\eps,1}(\bullet|X_{1-\eps}) \Big| R^{\nu,x,y}_{\eps,1}(\bullet|X_{1-\eps})\Big) \Big].
\end{align}

\paragraph{The entropy of $\boldsymbol{Q}$.}
We can do the same computations for $Q$ instead of $P$. In that case:
\begin{itemize}
	\item The formulae \eqref{eq.disintegration_endpoints} and \eqref{eq.disintegration_eps} are exactly the same, replacing the letter $P$ by the letter $Q$.
	\item The inequality \eqref{eq.entropy_wrt_bridge} is in the other sense (because of formula \eqref{eq.entropy_converse_BB} of Lemma \ref{lem.noisy_superposition_principle}), and $\H_\nu(\rho^{x,y}, c^{x,y})$ is replaced by $\H_\nu(\widetilde{\rho}^{x,y}, \widetilde{c}^{x,y})$:
	\begin{multline}
	\label{eq.entropy_wrt_bridge_Q}
	\nu H(Q_{\eps}^{x,y} | R^{\nu,x,y}_\eps) \leq \nu \frac{H(\rho_\eps^{x,y} | \Leb) + H( \rho_{1-\eps}^{x,y} | \Leb)}{2} \\- \nu \int \log f^{\nu,x,y}_\eps \D \rho_\eps^{x,y} - \nu \int \log g_\eps^{\nu,x,y}\D \rho_{1-\eps}^{x,y} + \H_\nu(\widetilde{\rho}^{x,y}, \widetilde{c}^{x,y}).
	\end{multline}
	(Recall that $\widetilde{\rhorho}$ and $\rhorho$ coincide at time $t=\eps$ and at time $t=1-\eps$.)
	\item As thanks to \eqref{eq.def_Qxy}: 
	\begin{equation*}
	Q^{x,y}(\bullet | X_{[\eps,1-\eps]}) = P_{\eps,0}^{x,y}(\bullet | X_\eps) \otimes P_{\eps,1}^{x,y}(\bullet | X_{1-\eps}),
	\end{equation*}
	we get an equality in \eqref{eq.P_not_Markov_apriori}:
	\begin{align}
	\notag \E&_{Q^{x,y}}\Big[H\Big(Q^{x,y}(\bullet | X_{[\eps,1-\eps]}) \Big| R^{\nu,x,y}(\bullet | X_{[\eps,1-\eps]})\Big)\Big]\\
	\notag &= \E_{Q^{x,y}}\left[ H\Big( P_{\eps,0}^{x,y}(\bullet | X_\eps) \otimes P_{\eps,1}^{x,y}(\bullet | X_{1-\eps})  \Big| R^{\nu,x,y}(\bullet | X_{[\eps,1-\eps]})\Big)\right]\\
	\notag &= \E_{Q^{x,y}}  \Big[  H\Big( P^{x,y}_{\eps,0}(\bullet|X_\eps) \Big| R^{\nu,x,y}_{\eps,0}(\bullet | X_\eps) \Big) \Big] +   \E_{Q^{x,y}} \Big[ H\Big(P^{x,y}_{\eps,1}(\bullet|X_{1-\eps}) \Big| R^{\nu,x,y}_{\eps,1}(\bullet|X_{1-\eps})\Big) \Big]\\
	\label{eq.Q_Markov} &= \E_{P^{x,y}}  \Big[  H\Big( P^{x,y}_{\eps,0}(\bullet|X_\eps) \Big| R^{\nu,x,y}_{\eps,0}(\bullet | X_\eps) \Big) \Big] +   \E_{P^{x,y}} \Big[ H\Big(P^{x,y}_{\eps,1}(\bullet|X_{1-\eps}) \Big| R^{\nu,x,y}_{\eps,1}(\bullet|X_{1-\eps})\Big) \Big].
	\end{align} 
	(The third line follows easy computations using the Markov property:
	\begin{equation*}
	R^{\nu,x,y}(\bullet | X_{[\eps,1-\eps]}) = R^{\nu,x,y}_{\eps,0}(\bullet | X_\eps)\otimes R^{\nu,x,y}_{\eps,1}(\bullet|X_{1-\eps}),
	\end{equation*}
	and the last one follows from the fact that the marginals of $P^{x,y}$ and $Q^{x,y}$ coincide at time $t = \eps$ and $1-\eps$.)
\end{itemize}
Gathering the formulae \eqref{eq.disintegration_endpoints}, \eqref{eq.disintegration_eps} for $P$ and $Q$, and \eqref{eq.entropy_wrt_bridge}, \eqref{eq.P_not_Markov_apriori}, \eqref{eq.entropy_wrt_bridge_Q} and \eqref{eq.Q_Markov}, we get:
\begin{equation}
\label{eq.comparison_entropies_P_Q}
\Hbar_{\nu}(Q) - \Hbar_{\nu}(P) \leq \HH_\nu(\widetilde{\rhorho}, \widetilde{\cc}) - \HH_\nu(\rhorho, \cc).
\end{equation}

\paragraph{Conclusion.}
Using \eqref{eq.P_sol}, we get as announced:
\begin{equation*}
\HH_\nu(\widetilde{\rhorho}, \widetilde{\cc}) \geq \HH_\nu(\rhorho, \cc),
\end{equation*}
or in other terms, $(\rhorho, \cc)$ is the solution of $\MBro_\nu(\rhorho_\eps, \rhorho_{1-\eps})$.

Remark that in the specific case when $(\widetilde{\rhorho}, \widetilde{\cc}) = (\rhorho, \cc)$, we get $\Hbar_{\nu}(Q) \leq \Hbar_{\nu}(P)$, which is compatible with \eqref{eq.P_sol} if and only if $\Hbar_{\nu}(Q) = \Hbar_{\nu}(P)$. Hence, in that case, by uniqueness of minimizers in $\Bro_\nu(\gamma)$, $P=Q$. It means that inequalities \eqref{eq.entropy_wrt_bridge} and \eqref{eq.P_not_Markov_apriori} are in fact equalities. We recover the known fact that for $\gamma$-almost all $(x,y)$, $P^{x,y}$ is Markovian, see \cite[Section 3]{arn17}. \qed
\subsection{Proof of Lemma \ref{lem.time_restriction}}
\label{subsec.proof_lemma_time_restriction}
First of all, because the Markov property of the Brownian motion $R^\nu$, the laws $R^\nu_\eps$ and $R_\eps^{\nu,x,y}$ have the same bridges:
\begin{equation*}
R^{\nu,x,y}\mbox{-almost everywhere}, \quad R_\eps^{\nu,x,y}(\bullet | X_\eps, X_{1-\eps}) = R_\eps^\nu(\bullet | X_\eps,  X_{1-\eps}).
\end{equation*}
In particular, $R_\eps^{\nu,x,y} \ll R^\nu_\eps$ if and only if $(X_\eps, X_{1-\eps})\pf R_\eps^{\nu,x,y} \ll (X_\eps, X_{1-\eps})\pf R^\nu_\eps$, and in that case:
\begin{equation}
\label{eq.same_bridges}
\frac{\D R^{\nu,x,y}_\eps}{\D R^\nu_\eps\phantom{{}^{,x,y}}} = 	\frac{\D \, (X_\eps, X_{1-\eps})\pf R^{\nu,x,y}_\eps}{\D\, (X_\eps, X_{1-\eps})\pf R^\nu_\eps\phantom{{}^{,x,y}}}\circ (X_\eps, X_{1-\eps}).
\end{equation}

In this proof, we will call:
\begin{equation*}
R^\nu_{\eps, 1-\eps} := (X_\eps, X_{1-\eps})\pf R^\nu , \quad 	R^{\nu,x,y}_{\eps, 1-\eps}:=  (X_\eps, X_{1-\eps})\pf R^{\nu,x,y}, \quad R^{\nu}_{0,\eps,1-\eps,1} := (X_0,X_\eps, X_{1-\eps},X_1)\pf R^\nu.
\end{equation*}
Let $(\tau^\nu_s)_{s\geq 0}$ be the heat flow of diffusivity $\nu$ on the torus \emph{i.e.} the solution to:
\begin{equation*}
\left\{  
\begin{gathered}
\partial_s \tau^\nu_s = \frac{\nu}{2} \Delta \tau^\nu_s, \\
\tau^\nu_0 = \delta_0.
\end{gathered}
\right.
\end{equation*}
Since the Brownian motion $R^\nu$ is a Markov process of generator $\nu / 2 \Delta$, the density of $R^{\nu}_{0,\eps, 1-\eps, 1}$ has the following Radon-Nikodym derivative with respect to the measure $\Leb^{\otimes 4}$:
\begin{equation*}
\frac{\D R^{\nu}_{0,\eps, 1-\eps, 1} }{\D \Leb^{\otimes 4}}(a,b,c,d) = \tau^\nu_{\eps}(b-a) \times \tau^\nu_{1-2\eps}(c-b) \times \tau^\nu_\eps(d-c). 
\end{equation*}
So by classical results concerning the behaviour of Radon-Nikodym derivatives towards conditionings:
\begin{align}
\notag\frac{\D R^{\nu,x,y}_{\eps, 1-\eps} }{\D \Leb^{\otimes 2}}(b,c) &= \frac{\tau^\nu_{\eps}(b-x) \times \tau^\nu_{1-2\eps}(c-b) \times \tau^\nu_\eps(y-c) }{ \int \tau^\nu_{\eps}(b'-x) \times \tau^\nu_{1-2\eps}(c'-b') \times \tau^\nu_\eps(y-c') \D b' \D c' } \\
\label{eq.density_bridge} &= \frac{\tau^\nu_{\eps}(b-x) \times \tau^\nu_{1-2\eps}(c-b) \times \tau^\nu_\eps(y-c)}{\tau^\nu_1(y-x)}.
\end{align}
(The second equality is deduced from the semi-group property of $(\tau^\nu_s)$.) On the other hand, we have:
\begin{equation}
\label{eq.density_brownian} \frac{\D R^{\nu}_{\eps, 1-\eps} }{\D \Leb^{\otimes 2}}(b,c) = \tau^\nu_{1-2\eps}(c-b).
\end{equation}
Gathering formulae \eqref{eq.density_bridge} and \eqref{eq.density_brownian}, we get:
\begin{equation*}
\frac{\D R^{\nu,x,y}_{\eps, 1-\eps} }{\D R^{\nu}_{\eps, 1-\eps} }(b,c) = \frac{\tau^\nu_{\eps}(b-x) \times \tau^\nu_\eps(y-c)}{\tau^\nu_1(y-x)}.
\end{equation*}
Plugging this identity into \eqref{eq.same_bridges}, we get \eqref{eq.bridge_sol_schro} with:
\begin{equation*}
f^{\nu,x,y}_\eps(b) := \frac{\tau^\nu_{\eps}(b-x)}{\sqrt{\tau^\nu_1(y-x)}} \qquad \mbox{and} \qquad g_\eps^{\nu,x,y}(c) := \frac{\tau^\nu_{\eps}(y-c)}{\sqrt{\tau^\nu_1(y-x)}}.
\end{equation*}

Then, \eqref{eq.entropy_bridge} just follow from the following easy computations:
\begin{align*}
H(Q_\eps | R^\nu_\eps) &= \E_{Q_\eps}\left[ \log \left( \frac{\D Q_\eps}{\D R^\nu_\eps} \right)\right]\\
&= \E_{Q_\eps}\left[ \log \left(\frac{\D Q_\eps}{\D R^{\nu,x,y}_\eps}\right) \right] + \E_{Q_\eps}\left[ \log\left( \frac{\D R^{\nu,x,y}_\eps}{\D R^\nu_\eps} \right)\right]\\
&= H(Q_\eps | R^{\nu,x,y}_\eps) + \E_{Q_\eps}\left[ \log f^{\nu,x,y}(X_\eps) \right] + \E_{Q_\eps}\left[ \log g^{\nu,x,y}(X_{1-\eps}) \right].
\end{align*}
The result follows easily. \qed

\subsection{Proof of Lemma \ref{lem.noisy_superposition_principle}}
\label{subsec.proof_lemma_noisy_superposition}

The proof follows closely the one of \cite[Theorem 3.4]{ambrosio2014continuity}. We take $(\rho,c)$ as in the statement of the lemma and $(\tau_\eps)_{\eps>0}$ a convolution kernel, everywhere positive. For a given $\eps>0$, we define:
\begin{equation*}
\rho^{\eps} := \rho \ast \tau_{\eps} \quad \mbox{and} \quad c^{\eps} := \frac{(\rho c) \ast \tau_{\eps}}{\rho^{\eps}}.
\end{equation*}
With this definition, $(\rho^\eps, c^\eps)$ is clearly a solution to the continuity equation and the following inequality is classical (see formula (3.5) in \cite{ambrosio2014continuity} with $\Theta = |\bullet|^2/2$):
\begin{equation}
\label{eq.A_decreases}
\A(\rho^\eps, c^\eps) \leq \A(\rho, c).
\end{equation}
(Recall that $\A$ is defined by formula \eqref{eq.def_action}.) Moreover, calling:
\begin{equation*}
w := \frac{\nu}{2} \nabla \log \rho \quad \mbox{and} \quad w^\eps := \frac{\nu}{2} \nabla \log \rho^\eps,
\end{equation*}
we have:
\begin{equation*}
w^{\eps} := \frac{\nu}{2} \frac{\nabla \rho^{\eps}}{\rho^{\eps}} = \frac{\nu}{2} \frac{(\nabla \rho) \ast \tau_{\eps}}{\rho^{\eps}}= \frac{(w \rho) \ast \tau_{\eps}}{\rho^{\eps}},
\end{equation*}
which means that $w^\eps$ is obtained from $w$ in the same way as $c^\eps$ is obtained from $c$. In particular,
\begin{equation}
\label{eq.F_decreases}
\nu^2 \F(\rho^\eps) = \A(\rho^\eps, w^\eps) \leq \A(\rho,w) = \nu^2 \F(\rho).
\end{equation}
(Recall that $\F$ is defined by \eqref{eq.def_Fischer}.) Gathering \eqref{eq.A_decreases} and \eqref{eq.F_decreases}, we get:
\begin{equation}
\label{eq.Hnu_decreases}
\H_\nu(\rho^\eps, c^\eps) \leq \H_\nu(\rho,c).
\end{equation}

Finally, we also get easily:
\begin{equation}
\label{eq.H_decreases}
H(\rho^\eps_0 | \Leb) \leq H(\rho_0 | \Leb) \quad \mbox{and} \quad H(\rho^\eps_1 | \Leb) \leq H(\rho_1 | \Leb).
\end{equation}
We can suppose that the entropies $ H(\rho_0 | \Leb)$ and $H(\rho_1 | \Leb)$ are finite\footnote{In fact, it is always the case, because we could show that these quantities are controlled by $\H_\nu(\rho,c)$, see \cite[Remark A.3]{baradat2018small}.}, because if they are not, Lemma \ref{lem.noisy_superposition_principle} reduces to \cite[Theorem 3.4]{ambrosio2014continuity}.

At this level of regularity, we can define $Q^\eps$ the (unique) law of the solution to the stochastic differential equation:
\begin{equation*}
\D X_t = v^{\eps}(t,X_t) \D t + \nu \D B_t, 
\end{equation*}
starting from $\rho^\eps_0$, where $v^\eps := c^\eps + w^\eps$, and where $B$ is a standard Brownian motion. For $t \in[0,1]$, we call $\widetilde{\rho}{}^\eps_t := X_t {}\pf Q^\eps$, the density of $Q^\eps$ at time $t$. Because $(\rho^{\eps}, c^\eps)$ is a solution to the continuity equation, by definition of $v^\eps$:
\begin{equation*}
\partial_t \rho^\eps_t + \Div(\rho^\eps v^\eps) = \frac{\nu}{2} \Delta \rho^\eps.
\end{equation*}
But by a standard application of the It\^o formula, we also have: 
\begin{equation*}
\partial_t \widetilde{\rho}{}^\eps_t + \Div(\widetilde{\rho}{}^\eps v^\eps) = \frac{\nu}{2} \Delta \widetilde{\rho}^\eps.
\end{equation*}
Consequently, $\rho^\eps$ and $\widetilde{\rho}^\eps$ are two solutions to the same parabolic equation with smooth coefficients, and with the same initial condition. So they coincide.

Thanks to formula \eqref{eq.entropie_girsanov} (here $\overrightarrow{b_t} = v^\eps(t, X_t)$), $H(Q^\eps | R^\nu)< + \infty$. So by Theorem \ref{thm.follmer}, the osmotic velocity of $Q^\eps$ is $\nu/2 \nabla \log \rho^\eps$, and by \eqref{eq.def_courant_osmotique}, its current velocity at time $t$ is $Q^\eps$-almost everywhere:
\begin{equation*}
v^\eps(t,X_t) - \frac{\nu}{2} \nabla \log \rho^\eps(t, X_t) = c^\eps(t,X_t).
\end{equation*}
In particular, thanks to \eqref{eq.entropie_girsanov_symetrique} and \eqref{eq.drift_markov}, we have:
\begin{equation}
\label{eq.estim_HbarP}
\Hbar_{\nu}(Q^\eps) = \nu \frac{H(\rho^\eps_0 | \Leb) + H(\rho^\eps_1 | \Leb)}{2} + \H_\nu(\rho^\eps, c^\eps) \leq \nu \frac{H(\rho_0 | \Leb) + H(\rho_1 | \Leb)}{2} + \H_\nu(\rho,c),
\end{equation}
where the last inequality is obtained thanks to \eqref{eq.Hnu_decreases} and \eqref{eq.H_decreases}. But $\Hbar_{\nu}$ has compact sublevels for the topology of narrow convergence, so we can find $Q$ a limit point of $(Q^\eps)_{\eps>0}$. 

The density of $Q$ is clearly $\rho$ (the density of a law is continuous with respect to narrow convergence). By lower semi-continuity of $\Hbar_{\nu}$, passing to the limit in \eqref{eq.estim_HbarP}, we get:
\begin{equation*}
\Hbar_{\nu}(Q) \leq \nu\frac{H(\rho_0 | \Leb) + H(\rho_1 | \Leb)}{2} + \H_\nu(\rho,c).
\end{equation*}
Hence, the result. \qed

\section{Existence of the pressure in the standard problem \sf{Br\"o}}
\label{sec.existence_Bro}
We are now ready to prove Theorem \ref{thm.existence_pressure_process}. The structure of the proof is the same as the one of \ref{thm.existence_pressure_MBro}, so we only treat the details of the parts that differ.

Given a bistochastic $\gamma$ satisfying condition \eqref{eq.condition_existence_process} and $\varphi \in \mathcal{E}_0$ with compact support in $(0,1)\times \T^d$, we define a new problem prescribing the density $(1 + \varphi)$ instead of $\Leb$ in $\Bro_\nu(\gamma)$, as in Problem~\ref{pb.Mbro_density} in the case of $\MBro$. We call $\H{}^*_\nu(1 + \varphi)$ the optimal value of $\Hbar_{\nu}$ in this new problem. As in Lemma~\ref{lem.Hnustar_convex}, $\H{}^*_\nu$ is convex and lower semi-continuous for the topology of $\mathcal{E}_0$.

Let us prove that $\eps$ being fixed, there exists $C>0$\footnote{Here, $C$ may depend on the dimension, $\eps$, $\rhorho_\eps$, $\rhorho_{1-\eps}$ and $\nu$. Contrary to before, we do not follow its dependence with respect to $\nu$.} such that for all $\varphi \in \mathcal{E}$ with $\Norm(\varphi) \leq 1/2$ and whose support is included in $(\eps, 1-\eps) \times \T^d$,
\begin{equation*}
\Hbar_\nu*(1 + \varphi) \leq C,
\end{equation*}
where $C$ only depend on the dimension, $\eps$ and $(\rhorho_\eps, \rhorho_{1-\eps})$

We define $(P^{x,y}_\eps, \rho^{x,y}, c^{x,y})_{(x,y) \in \T^d \times \T^d}$ as given by Theorem \ref{thm.lien_Bro_MBro}. Recall that by Theorem \ref{thm.lien_Bro_MBro}, $(\rhorho, \cc) = (\rho^{x,y}, c^{x,y})_{(x,y) \in \T^d \times \T^d}$ is a solution of $\MBro_\nu(\rhorho_\eps, \rhorho_{1-\eps})$ between the times $\eps$ and $1-\eps$. Then, we build from $(\rhorho, \cc)$ a competitor $(\widetilde{\rhorho}, \widetilde{\cc}) = (\widetilde{\rho}{}^{x,y}, \widetilde{c}{}^{x,y})$ for $\MBro_\nu(1 + \varphi)$ as defined in Problem \ref{pb.Mbro_density}, between the endpoints $\rhorho_\eps$ and $\rhorho_{1-\eps}$, and between the times $\eps$ and $1 - \eps$, as in the proof of Lemma~\ref{lem.Hnu_bounded}. From this proof, we have:
\begin{equation}
\label{eq.estim_HH}
\HH_\nu(\widetilde{\rhorho}, \widetilde{\cc}) \leq C ,
\end{equation}
where $C$ does not depend on $\varphi$.

Finally, we consider $Q$, build from $(\widetilde{\rhorho}, \widetilde{\cc})$ as in the proof of Theorem \ref{thm.lien_Bro_MBro}. By \eqref{eq.comparison_entropies_P_Q}, we have:
\begin{align}
\label{eq.comparison_entropies_P_Q_H*}\Hbar{}^*_\nu(1 + \varphi) &\leq \Hbar_\nu(Q) = \Hbar_\nu(P) + \HH_\nu(\widetilde{\rhorho}, \widetilde{\cc}) - \HH_\nu(\rhorho, \cc) \\
\notag &\leq \Hbar_{\nu}(P) + \HH_\nu(\widetilde{\rhorho}, \widetilde{\cc}) && \mbox{because }\HH_\nu \geq 0,\\
\notag &\leq \Hbar{}^*_\nu(\Leb) + C  && \mbox{by \eqref{eq.estim_HH}},\\
\notag &\leq C,
\end{align}
taking a larger $C$, but still independent of $\varphi$ in the last line.

We conclude that for all $\eps \in (0,1/2)$, there exists $p_\eps \in \mathcal{D}'((\eps, 1-\eps) \times \T^d)$ such that for all $\varphi \in \mathcal{E}_0$ with compact support in $(\eps, 1-\eps) \times \T^d$,
\begin{equation*}
\H{}^*_\nu(1 + \varphi) \geq \H{}^*_\nu(\Leb) + \cg p_\eps, \varphi \cd_{\mathcal{E}_0', \mathcal{E}_0}
\end{equation*}

We deduce from formula \eqref{eq.comparison_entropies_P_Q_H*} that $p_\eps$ is the pressure field in $\MBro_\nu(\rhorho_\eps, \rhorho_{1-\eps})$, so that by Lemma \ref{lem.charact_p}, $p_\eps$ is unique, and given by formula \eqref{eq.formule_p_MBro}. At last, by footnote \ref{footnote.restriction_time} and formula \eqref{eq.formule_p_MBro}, if $\eps_1 < \eps_2$, then $p_{\eps_2}$ is the restriction of $p_{\eps_1}$ to the set of times $(\eps_2, 1-\eps_2)$. So we end-up with a unique distribution $p$ satisfying the properties announced in the statement of Theorem \ref{thm.existence_pressure_process}.
\qed
\section{A formal way to derive the equation for the pressure}
\label{section.formal_computations}

Recall that in the case of incompressible optimal transport, if $P$ is a solution, and if $p$ is its pressure field that we suppose to be sufficiently regular, then for all $\eps \in (0,1/2)$, $P$-almost all curve $\omega$ is a minimizer of the Lagrangian:
\begin{equation*}
\int_\eps^{1-\eps} \Big\{ \frac{|\dot{\omega}_t|}{2}^2 - p(t,\omega_t) \Big\} \D t
\end{equation*}
among the set of curves whose positions at time $\eps$ and $1-\eps$ are $\omega_\eps$ and $\omega_{1-\eps}$ respectively.

In the case of the Br\"odinger problem, if $P$ is the solution of $\Bro_\nu(\gamma)$ and if $p$ is its pressure field, the corresponding expected result would be as follows. For $\gamma$-almost all $(x,y)$, $P^{x,y}_\eps$ as defined by formula~\eqref{eq.def_Pxyeps} should be the solution of the Schr\"odinger problem in the potential $p$, aiming at minimizing:
\begin{equation*}
\nu H(P|R^\nu) - \E_P\left[  \int p(t,X_t) \D t \right]
\end{equation*}
between the times $\eps$ and $1-\eps$, under constraints $X_\eps {}\pf P = X_\eps {}\pf P^{x,y}$ and  $X_{1-\eps} {}\pf P = X_{1-\eps} {}\pf P^{x,y}$. But in that case, it is known (see for instance \cite[Section 4.B]{zambrini1986variational}) that calling $\rho^{x,y}$ the density of $P^{x,y}$, $c^{x,y}$ its current velocity, and $w^{x,y} := \nu/2 \nabla \log \rho^{x,y}$, then $(\rho^{x,y}, c^{x,y}, w^{x,y})$ solves the following equations:
\begin{equation*}
\left\{
\begin{gathered}
\partial_t \rho^{x,y} + \Div (\rho^{x,y}c^{x,y}) = 0,\\
\partial_t c^{x,y} + (c^{x,y} \cdot \nabla)c^{x,y} + (w^{x,y} \cdot \nabla) w^{x,y} + \frac{\nu}{2} \boldsymbol{\Delta} w^{x,y} = - \nabla p.
\end{gathered}
\right.
\end{equation*}
(The notation $\boldsymbol{\Delta} w^{x,y}$ stands for the Laplacian operator computed coordinate by coordinate.) The second equation is reminiscent of the classical one:\footnote{usually written under the form of a Hamilton-Jacobi equation for $\theta^{x,y}$ satisfying $c^{x,y} = \nabla \theta^{x,y}$, see \cite{benamou2017variational}.}
\begin{equation*}
\partial_t c^{x,y} + (c^{x,y} \cdot \nabla)c^{x,y} = - \nabla p
\end{equation*}
for the velocity field in optimal transport with potential, plus osmotic terms of order $\nu^2$. If we multiply this equation by $\rho^{x,y}$ and if we use the identities:
\begin{gather*}
\bDiv ( c^{x,y} \otimes c^{x,y} \rho^{x,y} ) = \rho^{x,y} (c^{x,y} \cdot \nabla ) c^{x,y} + c^{x,y} \Div(\rho^{x,y} c^{x,y}),\\
\bDiv ( w^{x,y} \otimes w^{x,y} \rho^{x,y} ) = \frac{\nu}{4}^2 \boldsymbol{\Delta} \nabla \rho^{x,y}  - \frac{\nu}{2} \rho^{x,y} \boldsymbol{\Delta} w^{x,y} - \rho^{x,y} (w^{x,y} \cdot \nabla) w^{x,y},
\end{gather*} 
we get the following equation for the momentum:
\begin{gather*}
\partial_t(\rho^{x,y} c^{x,y}) + \bDiv\Big( \{ c^{x,y} \otimes c^{x,y} - w^{x,y} \otimes w^{x,y} \} \rho^{x,y} \Big) + \frac{\nu}{4} \boldsymbol{\Delta} \nabla \rho^{x,y} \\
= - \rho^{x,y} \nabla p + c^{x,y} \cancel{(\partial_t \rho^{x,y} + \Div(\rho^{x,y} c^{x,y}))} = -\rho^{x,y} \nabla p.
\end{gather*}
If we integrate with respect to $\gamma$, because of incompressibility, the $\boldsymbol{\Delta} \nabla$ term cancels and the coefficient in front of $\nabla p$ becomes $1$. So we get:
\begin{equation*}
\partial_t \left( \int \rho^{x,y} c^{x,y} \D \gamma(x,y) \right) + \bDiv \left( \int \{ c^{x,y} \otimes c^{x,y} - w^{x,y} \otimes w^{x,y} \} \rho^{x,y} \D \gamma(x,y)  \right) = - \nabla p.
\end{equation*}
This is exactly formula \eqref{eq.formule_p_MBro} derived earlier, with $\I = \T^d \times \T^d$ and $\m = \gamma$, which is coherent with the fact that we observed in Theorem \ref{thm.lien_Bro_MBro} that $(\rhorho, \cc) = (\rho^{x,y}, c^{x,y})_{(x,y) \in \T^d \times \T^d}$ is the solution to $\MBro_\nu$ when localized in times, with respect to its own endpoints.

It is likely that just as in the incompressible optimal transport case, regularity estimates for the pressure field would make it possible to justify rigorously these computations, but we did not pursue in this direction.
\paragraph{Acknowledgments} I would like to thank Christian L\'eonard for the many fruitful discussions during my master thesis that he supervised and still after. I also would like to thank him for having shared with me the preliminary versions of \cite{arn17}. This work is part of my PhD thesis supervised by Yann Brenier and Daniel Han-Kwan and I also would like to thank both of them.

 \bibliography{../../../biblio/bibliographie}
 \bibliographystyle{plain}
  \end{document}